\documentclass[letterpaper,10pt]{article}

\usepackage[margin=0.5in]{geometry}

\usepackage{fullpage}
\usepackage{enumerate}
\usepackage{authblk}
\usepackage[colorinlistoftodos]{todonotes}
\usepackage{verbatim}%comments out anything between \begin{comment} and \end{comment}
\usepackage{section, amsthm, textcase, setspace, amssymb, lineno, amsmath, amssymb, amsfonts, latexsym, fancyhdr, longtable, ulem, mathtools}
\usepackage{epsfig, graphicx, pstricks,pst-grad,pst-text,tikz, tkz-berge,colortbl}
\usepackage{epsf}
\usepackage{graphicx, color}
\usepackage{float}
\usepackage[rflt]{floatflt}
\usepackage{amsfonts}
\usepackage{latexsym,enumitem}
\usetikzlibrary{fit,matrix,positioning}
\usepackage{pdflscape}
\usetikzlibrary{decorations.pathreplacing}
\usepackage{mathrsfs}

\definecolor{darkgreen}{rgb}{0, 0.5, 0}

\newtheorem{theorem}{Theorem}
\newtheorem{lemma}{Lemma}

\newtheorem{definition}{Definition}

\newtheorem{Ex}{Example}

\newtheorem*{theorem*}{Theorem}
\newtheorem{remark}{Remark}

\newcommand{\ind}{{\rm ind \hspace{.1cm}}}

\begin{document}

\title{Toral posets and the binary spectrum property}

\author[*]{Vincent E. Coll, Jr.}
\author[*]{Nicholas W. Mayers}

\affil[*]{Department of Mathematics, Lehigh University, Bethlehem, PA, 18015}

\maketitle

%linenumbers

\begin{abstract}
\noindent
We introduce a family of posets which generate Lie poset subalgebras of $A_{n-1}=\mathfrak{sl}(n)$ whose index
can be realized topologically.  In particular, if $\mathcal{P}$ is such a \textit{toral poset}, then it has a simplicial realization which is homotopic to a wedge sum of $d$ one-spheres, where $d$ is the index of the corresponding type-A Lie poset algebra $\mathfrak{g}_A(\mathcal{P})$. Moreover, when $\mathfrak{g}_A(\mathcal{P})$ is Frobenius, its spectrum is \textit{binary}; that is, consists of an equal number of 0's and 1's. We also find that all Frobenius, type-A Lie poset algebras corresponding to a poset whose largest totally ordered subset is of cardinality at most three have a binary spectrum.
\end{abstract}



%\tableofcontents
%%%%%%%%%%%%%%%%%%%%%%%%%%%%%%%%%%%%%%%%%%%%%%%%%%%
%%%%%%%%%%%%%%%%%%%%%%%%%%%%%%%%%%%%%%%%%%%%%%%%%%%
\section{Introduction}
%%%%%%%%%%%%%%%%%%%%%%%%%%%%%%%%%%%%%%%%%%%%%%%%%%%
%%%%%%%%%%%%%%%%%%%%%%%%%%%%%%%%%%%%%%%%%%%%%%%%%%%

In \textbf{\cite{CG}}, Coll and Gerstenhaber introduce the notion of a ``Lie poset algebra".  These Lie algebras are subalgebras of $A_{n-1}=\mathfrak{sl}(n)$ which lie between the subalgebras of upper-triangular and diagonal matrices.  We refer to such Lie algebras as \textit{type-A Lie poset algebras} and find that they are naturally associated to the incidence algebras of posets \textbf{\cite{Ro}} (see Section~\ref{sec:lieposet}). In \textbf{\cite{CM}}, the current authors establish formulas for the index of type-A Lie poset algebras corresponding to posets whose largest totally ordered subset is of cardinality at most three. The authors further characterize such posets which correspond to type-A Lie poset algebras with index zero. In this article, we initiate an investigation into the form of the spectrum of such Lie algebras.

\bigskip
Formally, the index of a Lie algebra $\mathfrak{g}$ is defined as 
\[\ind \mathfrak{g}=\min_{F\in \mathfrak{g^*}} \dim  (\ker (B_F)),\]

\noindent where $B_F$ is the skew-symmetric \textit{Kirillov form} defined by $B_F(x,y)=F([x,y])$, for all $x,y\in\mathfrak{g}$. Of particular interest are those Lie algebras which have index zero, and are called \textit{Frobenius}.\footnote{Frobenius algebras are of  special interest in deformation and quantum group theory stemming from their connection with the classical Yang-Baxter equation (see \textbf{\cite{G1,G2}}).} A functional $F\in\mathfrak{g}^*$ for which $\dim\ker(B_F)=\ind(\mathfrak{g})=0$ is likewise called \textit{Frobenius}. Given a Frobenius Lie algebra $\mathfrak{g}$ and a Frobenius functional $F\in\mathfrak{g}^*$, the map $\mathfrak{g}\to\mathfrak{g}^*$ defined by $x\mapsto B_F(x,-)$ is an isomorphism. The inverse image of $F$ under this isomorphism, denoted $\widehat{F}$, is called a \textit{principal element} of $\mathfrak{g}$. In \textbf{\cite{Ooms}}, Ooms shows that the eigenvalues (and multiplicities) of $ad(\widehat{F})=[\widehat{F},-]:\mathfrak{g}\to\mathfrak{g}$ do not depend on the choice of principal element $\widehat{F}$ (see also \textbf{\cite{Prin}}). It follows that the spectrum of $ad(\widehat{F})$ is an invariant of $\mathfrak{g}$, which we call the \textit{spectrum} of $\mathfrak{g}$.\footnote{Our investigation into the spectral theory of type-A Lie poset algebras is largely motivated by the corresponding theory for seaweed algebras. \textit{Seaweed} (or \textit{biparabolic}) subalgebras of a complex semi-simple Lie algebra $\mathfrak{g}$ are intersections of two parabolic subalgebras whose sum is $\mathfrak{g}$ (see \textbf{\cite{Joseph,Panyushev1}}). It has been shown that the spectrum of seaweed subalgebras of the classical families of Lie algebras consists of an unbroken sequence of integers where the multiplicities of the eigenvalues form a symmetric distribution about one half (see \textbf{\cite{specD, specAB,unbroken}}).}

In this article, we introduce a family of posets which generate type-A Lie poset algebras whose index can be realized topologically. In particular, if $\mathcal{P}$ is such a ``toral poset" (see Definition~\ref{def:toral}), then it has a simplicial realization which is homotopic to a wedge sum of $\ind\mathfrak{g}_A(\mathcal{P})$ one-spheres (see Theorem~\ref{thm:wedge}). Moreover, when $\mathfrak{g}_A(\mathcal{P})$ is Frobenius, its spectrum is \textit{binary}; that is, consists of an equal number of 0's and 1's (see Theorem~\ref{thm:toralspec}). We also find that all Frobenius, type-A Lie poset algebras corresponding to a poset whose largest totally ordered subset is of cardinality at most three have a binary spectrum (see Theorem~\ref{thm:end}). Extensive calculations suggest that all Frobenius, type-A Lie poset algebras have a binary spectrum.

The structure of the paper is as follows.  In Section~\ref{sec:poset} we set the combinatorial definitions and notation related to posets and in Section~\ref{sec:lieposet} we formally introduce type-A Lie poset algebras.  Section~\ref{sec:indspec} deals with the determination of the form of a principal element for certain Frobenius, type-A Lie poset algebras.  Sections \ref{sec:toralpair}, \ref{sec:toralpos}, and \ref{sec:spec} deal with the main objects of interest: toral posets -- and their associated index and spectral theories. 

%%%%%%%%%%%%%%%%%%%%%%%%%%%%%%%%%%%%%%%%%%%%
\section{Posets}\label{sec:poset}

A \textit{finite poset} $(\mathcal{P}, \preceq_{\mathcal{P}})$ consists of a finite set $\mathcal{P}$ together with a binary relation $\preceq_{\mathcal{P}}$ which is reflexive, anti-symmetric, and transitive. When no confusion will arise, we simply denote a poset $(\mathcal{P}, \preceq_{\mathcal{P}})$ by $\mathcal{P}$, and $\preceq_{\mathcal{P}}$ by $\preceq$. Throughout, we let $\le$ denote the natural ordering on $\mathbb{Z}$. Two posets $\mathcal{P}$ and $\mathcal{Q}$ are \textit{isomorphic} if there exists an order-preserving bijection $\mathcal{P}\to\mathcal{Q}$.

Given a finite poset $\mathcal{P}$, let $Rel(\mathcal{P})$, $Ext(\mathcal{P})$, and $Rel_E(\mathcal{P})$ denote, respectively, 
the set of strict relations between elements of $\mathcal{P}$, the set of minimal and maximal elements of $\mathcal{P}$, and the set of strict relations between elements of $Ext(\mathcal{P})$. If $x,y\in\mathcal{P}$, $x\preceq y$ and there exists no $z\in \mathcal{P}$ satisfying $x,y\neq z$ and $x\preceq z\preceq y$, then $x\preceq y$ is a \textit{covering relation}. Covering relations are used to define a visual representation of $\mathcal{P}$ called the \textit{Hasse diagram} -- a graph whose vertices correspond to elements of $\mathcal{P}$ and whose edges correspond to covering relations (see, for example,  Figure~\ref{fig:Hasse} (a)). Extending the Hasse diagram of $\mathcal{P}$ by allowing all elements of $Rel(\mathcal{P})$ to define edges results in the \textit{comparability graph} of $\mathcal{P}$ (see, for example, Figure~\ref{fig:Hasse} (b)).

A totally ordered subset $S\subset\mathcal{P}$ is called a \textit{chain}. The \textit{height} of $\mathcal{P}$ is one less than the largest cardinality of a chain in $\mathcal{P}$. One can define a simplicial complex $\Sigma(\mathcal{P})$ by having chains of cardinality $n-1$ in $\mathcal{P}$ define the $n$-dimensional faces of $\Sigma(\mathcal{P})$ (see, for example, Figure~\ref{fig:Hasse} (c)). A subset $I\subset\mathcal{P}$ is an \textit{order ideal} if given $y\in\mathcal{P}$ such that there exists $x\in I$ satisfying $y\preceq x$, then $y\in I$. Similarly, a subset $F\subset\mathcal{P}$ is a \textit{filter} if given $y\in\mathcal{P}$ such that there exists $x\in F$ satisfying $x\preceq y$, then $y\in F$.

%(See the extended Example~\ref{ex:not} below).

\begin{Ex}\label{ex:not}
Consider the poset $\mathcal{P}=\{1,2,3,4\}$ with $1\preceq 2\preceq 3,4$, then we have $$Rel(\mathcal{P})=\{1\preceq 2,1\preceq 3,1\preceq 4,2\preceq 3,2\preceq 4\},~~ Ext(\mathcal{P})=\{1,3,4\},~~\text{and}~~ Rel_E(\mathcal{P})=\{1\preceq 3, 1\preceq 4\}.$$
%In Figure~\ref{fig:Hasse} below, we illustrate the (a) Hasse diagram and (b) comparability graph of $\mathcal{P}$ as well as (c) $\Sigma(\mathcal{P})$.

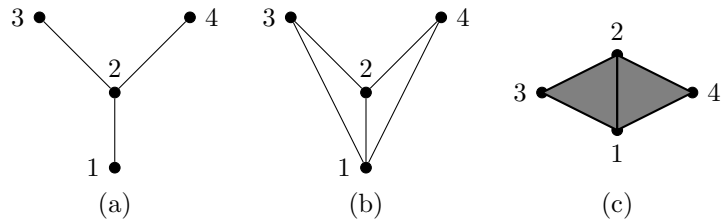
\begin{figure}[H]
$$\begin{tikzpicture}
	\node (1) at (0, 0) [circle, draw = black, fill = black, inner sep = 0.5mm, label=left:{$1$}]{};
	\node (2) at (0, 1)[circle, draw = black, fill = black, inner sep = 0.5mm, label=above:{$2$}] {};
	\node (3) at (-1, 2) [circle, draw = black, fill = black, inner sep = 0.5mm, label=left:{$3$}] {};
	\node (4) at (1, 2) [circle, draw = black, fill = black, inner sep = 0.5mm, label=right:{$4$}] {};
	\node (5) at (0,-0.5) {(a)};
    \draw (1)--(2);
    \draw (2)--(3);
    \draw (2)--(4);
\end{tikzpicture}\quad\begin{tikzpicture}
	\node (1) at (0, 0) [circle, draw = black, fill = black, inner sep = 0.5mm, label=left:{$1$}]{};
	\node (2) at (0, 1)[circle, draw = black, fill = black, inner sep = 0.5mm, label=above:{$2$}] {};
	\node (3) at (-1, 2) [circle, draw = black, fill = black, inner sep = 0.5mm, label=left:{$3$}] {};
	\node (4) at (1, 2) [circle, draw = black, fill = black, inner sep = 0.5mm, label=right:{$4$}] {};
	\node (5) at (0,-0.5) {(b)};
    \draw (1)--(2);
    \draw (2)--(3);
    \draw (2)--(4);
    \draw (1)--(3);
    \draw (1)--(4);
\end{tikzpicture}\quad\begin{tikzpicture}
	\node (1) at (0, 0) [circle, draw = black, fill = black, inner sep = 0.5mm, label=below:{$1$}]{};
	\node (2) at (-1, 0.5)[circle, draw = black, fill = black, inner sep = 0.5mm,label=left:{$3$}] {};
    \node (3) at (0, 1) [circle, draw = black, fill = black, inner sep = 0.5mm,label=above:{$2$}]{};
    \node (4) at (1, 0.5) [circle, draw = black, fill = black, inner sep = 0.5mm,label=right:{$4$}]{};
    \node (5) at (0,-1) {(c)};
    \draw [fill=gray,thick] (0, 0)--(-1, 0.5)--(0, 1)--(0, 0);
    \draw [fill=gray,thick] (0, 0)--(1, 0.5)--(0, 1)--(0, 0);
\end{tikzpicture}$$
\caption{Hasse diagram of $\mathcal{P}$, comparability graph of $\mathcal{P}$, and $\Sigma(\mathcal{P})$}\label{fig:Hasse}
\end{figure}

\noindent
Note that $\{1,2\}\subset \mathcal{P}$ is both an order ideal and a chain of $\mathcal{P}$, but not a filter, while 
$\{3,4\}\subset \mathcal{P}$ is a filter, but neither an order ideal nor a chain. 
\end{Ex}

\section{Lie poset algebras}\label{sec:lieposet}

Let $\mathcal{P}$ be a finite poset. The \textit{associative incidence} (or \textit{poset}) \textit{algebra} $A=A(\mathcal{P}, \textbf{k})$ is the span over $\textbf{k}$ of elements $E_{p_i,p_j}$, for $p_i\preceq p_j$, with multiplication given by setting $E_{p_i,p_j}E_{p_k,p_l}=E_{p_i,p_l}$ if $p_j=p_k$ and $0$ otherwise.  The \textit{trace} of an element $\sum c_{p_i,p_j}E_{p_i,p_j}$ is $\sum c_{p_i,p_i}$.  We can equip $A$ with the commutator bracket $[a,b]=ab-ba$, where juxtaposition denotes the product in $A$, to produce the \textit{Lie poset algebra} $\mathfrak{g}(\mathcal{P})
=\mathfrak{g}(\mathcal{P}, \textbf{k})$. 

\begin{Ex}\label{ex:gstargate}
If $\mathcal{P}$ is the poset of Example~\ref{ex:not}, then $\mathfrak{g}(\mathcal{P})$ is the span over \textup{$\textbf{k}$} of the elements of $$\{E_{1,1},E_{2,2},E_{3,3},E_{4,4},E_{1,2},E_{1,3},E_{1,4},E_{2,3},E_{2,4}\}.$$
\end{Ex}

\noindent
Restricting to the trace-zero elements of $\mathfrak{g}(\mathcal{P})$ results in the \textit{type-A Lie poset algebra} $\mathfrak{g}_A(\mathcal{P})$.

\begin{Ex}
If $\mathfrak{g}(\mathcal{P})$ is as in Example~\ref{ex:gstargate}, then $\mathfrak{g}_A(\mathcal{P})$ is the span over \textup{$\textbf{k}$} of the elements of $$\{E_{1,1}-E_{2,2},E_{2,2}-E_{3,3},E_{3,3}-E_{4,4},E_{1,2},E_{1,3},E_{1,4},E_{2,3},E_{2,4}\}.$$
\end{Ex}

\begin{remark}
Isomorphic posets correspond to isomorphic \textup(type-A\textup) Lie poset algebras.
\end{remark}

\begin{remark}
The definition of type-A Lie poset algebra given in this section can be seen to be equivalent to that found in \textup{\textbf{\cite{CG}}} as follows: if $|\mathcal{P}|=n$, then taking a linear extension of $\mathcal{P}$, i.e., an isomorphism $(\mathcal{P},\preceq_{\mathcal{P}})\to(\{1,\hdots,n\},\le)$, both $A$ and $\mathfrak{g}$ may be regarded as subalgebras of the algebra consisting of all $n \times n$ upper-triangular matrices over \textup{$\textbf{k}$}. Such a matrix representation is realized by replacing each basis element $E_{i,j}$ by the $n\times n$ matrix $M_{i,j}$ containing a 1 in the $i,j$-entry, and 0's elsewhere. The product between elements $E_{i,j}$ is then replaced by matrix multiplication between the $M_{i,j}$. Restricting to $\mathfrak{g}_A(\mathcal{P})$ yields a Lie subalgebra of the first classical family $A_{n-1}=\mathfrak{sl}(n)$ contained between the subalgebras of upper-triangular and diagonal matrices.
\end{remark}

\noindent
We refer to posets $\mathcal{P}$ for which $\mathfrak{g}_A(\mathcal{P})$ is Frobenius as \textit{Frobenius posets}.  

\section{Frobenius functionals and principal elements}\label{sec:indspec}

In this section, we develop a framework for analyzing the spectrum of Frobenius, type-A Lie poset algebras by determining the form of a particular principal element.

Given a finite poset $\mathcal{P}$ and $B=\sum b_{p,q}E_{p,q}\in\mathfrak{g}_A(\mathcal{P})$, define $E_{p,q}^*\in(\mathfrak{g}_A(\mathcal{P}))^*$, for $p,q\in\mathcal{P}$ satisfying $p\preceq q$, by $E^*_{p,q}(B)=b_{p,q}$. From any set $S$ consisting of ordered pairs $(p,q)$ of elements $p,q\in\mathcal{P}$ satisfying $p\preceq q$, i.e., $S\subset Rel(\mathcal{P})$, one can construct both a functional $F_S=\sum_{(p,q)\in S}E_{p,q}^*\in(\mathfrak{g}_A(\mathcal{P}))^*$ as well as a directed subgraph $\Gamma_{F_S}(\mathcal{P})$ of the comparability graph of $\mathcal{P}$. In \textbf{\cite{Prin}}, Gerstenhaber and Giaquinto refer to such a functional as \textit{small} if $\Gamma_{F_S}(\mathcal{P})$ is a spanning subtree of the comparability graph of $\mathcal{P}$. Note, if $F_S\in(\mathfrak{g}_A(\mathcal{P}))^*$ is a small functional, then $\Gamma_{F_S}(\mathcal{P})$ naturally partitions the elements of $\mathcal{P}$ into the following disjoint subsets: 
\begin{itemize}
    \item $U_{F_S}(\mathcal{P})$ consisting of all sinks in $\Gamma_{F_S}(\mathcal{P})$, 
    \item $D_{F_S}(\mathcal{P})$ consisting of all sources  in $\Gamma_{F_S}(\mathcal{P})$, and 
    \item $B_{F_S}(\mathcal{P})$ consisting of those vertices which are neither sinks nor sources in $\Gamma_{F_S}(\mathcal{P})$.
\end{itemize}

\begin{remark}\label{rem:hatalg}
In \textup{\textbf{\cite{Prin}}}, the authors establish a method of calculating $\widehat{F}_S$ for a small, Frobenius functional $F_S$. Their algorithm is equivalent to solving the following system of equations: if $\widehat{F}_S=\sum_{p\in\mathcal{P}}c_{p,p}E_{p,p}$, then
\begin{itemize}
    \item $c_{p,p}-c_{q,q}=1$ for $(p,q)\in S$, and
    \item $\sum_{p\in\mathcal{P}}c_{p,p}=0$.
\end{itemize}
\end{remark}

\noindent
Ongoing, we assume that every functional $F\in(\mathfrak{g}_A(\mathcal{P}))^*$ is of the form $F_S$ for some $S\subset Rel(\mathcal{P})$.

\begin{theorem}\label{thm:spec01}
Let $\mathcal{P}$ be a Frobenius poset. If $F\in(\mathfrak{g}_A(\mathcal{P}))^*$ satisfies the following conditions:
\begin{itemize}
    \item $F$ is small and Frobenius,
    \item $U_{F}(\mathcal{P})$ is a filter of $\mathcal{P}$,
    \item $D_{F}(\mathcal{P})$ is an ideal of $\mathcal{P}$, and
    \item $B_{F}(\mathcal{P})=\emptyset$,
\end{itemize} 
then $\widehat{F}=\sum_{p\in\mathcal{P}}c_{p,p}E_{p,p}$ satisfies 
\[c_{p,p} =  \begin{cases} 
      \frac{|U_{F_S}(\mathcal{P})|}{|\mathcal{P}|}, & p\in D_{F_S}(\mathcal{P}); \\
                                                &                        \\
      \frac{-|D_{F_S}(\mathcal{P})|}{|\mathcal{P}|}, & p\in U_{F_S}(\mathcal{P}).
   \end{cases}
\]
\end{theorem}
\begin{proof}
Assume $F=F_S$ for $S\subset Rel(\mathcal{P})$. To determine the form of $\widehat{F}_S$, we use the system of equations given in Remark~\ref{rem:hatalg}.

Let $p_1\in D_{F_S}(\mathcal{P})$ and $p_n\in U_{F_S}(\mathcal{P})$ with $(p_1,p_n)\in S$, i.e., $c_{p_1,p_1}-1=c_{p_n,p_n}$. Since $\Gamma_{F_S}(\mathcal{P})$ is connected, given $p\in D_{F_S}(\mathcal{P})$, there exists a path from $p_1$ to $p$ in $\Gamma_{F_S}(\mathcal{P})$. Assume that such a path is defined by the following sequence of vertices of $\Gamma_{F_S}$: $p_1=p_{i_0},p_{i_1},p_{i_2},\hdots,p_{i_{n-1}},p_{i_n}=p$. By our assumption that $B_{F_S}=\emptyset$, we must have $$p_1=p_{i_0}\preceq p_{i_1}\succeq p_{i_2}\preceq \hdots\preceq p_{i_{n-1}}\succeq p_{i_n}=p,$$ so that $$c_{p_{i_0},p_{i_0}}-c_{p_{i_1},p_{i_1}}=1$$ $$c_{p_{i_2},p_{i_2}}-c_{p_{i_1},p_{i_1}}=1$$ $$c_{p_{i_2},p_{i_2}}-c_{p_{i_3},p_{i_3}}=1$$ $$c_{p_{i_4},p_{i_4}}-c_{p_{i_3},p_{i_3}}=1$$ $$\vdots$$ $$c_{p_{i_n},p_{i_n}}-c_{p_{i_{n-1}},p_{i_{n-1}}}=1.$$  Solving the above equations, we find that $$c_{p_1,p_1}=c_{p_{i_0},p_{i_0}}=c_{p_{i_1},p_{i_1}}+1=c_{p_{i_2},p_{i_2}}=\hdots=c_{p_{i_{n-1}},p_{i_{n-1}}}+1=c_{p,p};$$ that is, $c_{p_1,p_1}=c_{p,p}$ for all $p\in D_{F_S}(\mathcal{P})$. Similarly, we find that $c_{p_n,p_n}=c_{p,p}$ for all $p\in U_{F_S}(\mathcal{P})$. Thus, $c_{p,p}=c_{p_1,p_1}-1$ for all $p\in U_{F_S}(\mathcal{P})$ and the condition $\sum_{p\in\mathcal{P}} c_{p,p}=0$ becomes $|\mathcal{P}|c_{p_1,p_1}=|U_{F_S}(\mathcal{P})|$. Therefore, $c_{p_1,p_1}=\frac{|U_{F_S}(\mathcal{P})|}{|\mathcal{P}|}$ and

\[c_{p,p} =  \begin{cases} 
      \frac{|U_{F_S}(\mathcal{P})|}{|\mathcal{P}|}, & p\in D_{F_S}(\mathcal{P}); \\
                                                &                        \\
      \frac{-|D_{F_S}(\mathcal{P})|}{|\mathcal{P}|}, & p\in U_{F_S}(\mathcal{P}).
   \end{cases}
\]
\end{proof}

\begin{remark}\label{rem:binbasis}
If $F\in(\mathfrak{g}_A(\mathcal{P}))^*$ satisfies the conditions of Theorem~\ref{thm:spec01}, then one obtains a canonical choice of basis for $\mathfrak{g}_A(\mathcal{P})$:
$$\mathscr{B}_{\mathcal{P},F}=\{E_{p,q}~|~p,q\in\mathcal{P},p\preceq_{\mathcal{P}}q\}\cup\{E_{p,p}-E_{q,q}~|~E^*_{p,q}\text{ is a summand of }F\}.
$$ This basis will prove useful in the analysis of the spectrum of Frobenius, type-A Lie poset algebras.
\end{remark}

\noindent
The following result is an immediate corollary to Theorem~\ref{thm:spec01}.

\begin{theorem}\label{thm:01}
If $\mathcal{P}$ is a Frobenius poset and $F\in(\mathfrak{g}_A(\mathcal{P}))^*$ satisfies the conditions of Theorem~\ref{thm:spec01}, then the spectrum of $\mathfrak{g}_A(P)$ consists of 0's and 1's.
\end{theorem}
\begin{proof}
To determine the spectrum of $\mathfrak{g}_A(\mathcal{P})$, we calculate the values $[\widehat{F},x]$, for $x\in\mathscr{B}_{\mathcal{P},F}$.
To start, for $x\in\{E_{p,p}-E_{q,q}~|~E^*_{p,q}\text{ is a summand of }F\}\subset\mathscr{B}_{\mathcal{P},F}$, we must have $[\widehat{F},x]=0\cdot x$. It remains to consider basis elements of the form $E_{p,q}\in \mathscr{B}_{\mathcal{P},F}$. The analysis of such basis elements breaks into three cases:
\\*

\noindent
\textbf{Case 1:} if $p,q\in U_{F}(\mathcal{P})$, then $$[\widehat{F},E_{p,q}]=\bigg(\frac{-|D_{F}(\mathcal{P})|}{|\mathcal{P}|}-\bigg(\frac{-|D_{F}(\mathcal{P})|}{|\mathcal{P}|}\bigg)\bigg)\cdot E_{p,q}=0\cdot E_{p,q}.$$

\noindent
\textbf{Case 2:} if $p,q\in D_{F}(\mathcal{P})$, then $$[\widehat{F},E_{p,q}]=\bigg(\frac{|U_{F}(\mathcal{P})|}{|\mathcal{P}|}-\frac{|U_{F}(\mathcal{P})|}{|\mathcal{P}|}\bigg)\cdot E_{p,q}=0\cdot E_{p,q}.$$ 

\noindent
\textbf{Case 3:} if $p\in D_{F}(\mathcal{P})$ and $q\in U_{F}(\mathcal{P})$, then $$[\widehat{F},E_{p,q}]=\bigg(\frac{|U_{F}(\mathcal{P})|}{|\mathcal{P}|}-\bigg(\frac{-|D_{F}(\mathcal{P})|}{|\mathcal{P}|}\bigg)\bigg)\cdot E_{p,q}=\bigg(\frac{|U_{F}(\mathcal{P})|+|D_{F}(\mathcal{P})|}{|\mathcal{P}|}\bigg)\cdot E_{p,q}=1\cdot E_{p,q}.$$ Thus, as $\mathscr{B}_{\mathcal{P},F}$ forms a basis for $\mathfrak{g}_A(\mathcal{P})$, the spectrum of $\mathfrak{g}_A(\mathcal{P})$ consists of 0's and 1's.
\end{proof}

\begin{remark}
Note that Theorem~\ref{thm:01} only provides information about the spectrum of $\widehat{F}$, but not the multiplicities of the eigenvalues.
\end{remark}

\noindent
In the next section, numerous examples of Frobenius posets will be given for which there exists a corresponding Frobenius functional with the properties listed in Theorems~\ref{thm:spec01} and~\ref{thm:01}.

\section{Toral-pairs}\label{sec:toralpair}
In this section, we introduce the notion of a \textit{toral-pair}, consisting of a Frobenius poset together with a certain Frobenius functional, and we give numerous examples of such pairs. The posets of such pairs will form the ``building blocks" used to construct the main objects of interest in this paper, toral posets.

\begin{definition}\label{def:toralpair}
Given a Frobenius poset $\mathcal{P}$ and a corresponding Frobenius functional $F\in(\mathfrak{g}_A(\mathcal{P}))^*$, we call $(\mathcal{P},F)$ a toral-pair if $\mathcal{P}$ satisfies
\begin{enumerate}[label={\textup{(\bfseries P\arabic*)}}]
    \item $|Ext(\mathcal{P})|=2$ or 3,
    \item $\mathfrak{g}_A(\mathcal{P})$ has a binary spectrum, and
    \item $\Sigma(\mathcal{P})$ is contractible,
\end{enumerate}
and $F$ satisfies
\begin{enumerate}[label={\textup{(\bfseries F\arabic*)}}]
    \item $F$ is small,
     \item $U_{F}(\mathcal{P})$ is a filter of $\mathcal{P}$, $D_{F}(\mathcal{P})$ is an ideal of $\mathcal{P}$, and $B_{F}(\mathcal{P})=\emptyset$,
    \item $\Gamma_{F}$ contains all edges between elements of $Ext(\mathcal{P})$, and
    \item $B\in\mathfrak{g}(\mathcal{P})\cap\ker(F)$ satisfies $E^*_{p,p}(B)=E^*_{q,q}(B)$, for all $p,q\in\mathcal{P}$, and $E^*_{p,q}(B)=0$, for all $p,q\in\mathcal{P}$ satisfying $p\preceq q$.
\end{enumerate}
\end{definition}

\begin{Ex}\label{ex:toralpairs}
The posets illustrated in Figure~\ref{fig:bb} can be paired with an appropriate functional to form a toral-pair \textup(see Theorems~\ref{thm:h1bb},~\ref{thm:h2bb}, and~\ref{thm:h3bb}\textup).

\begin{figure}[H]
$$\begin{tikzpicture}[scale=0.7]
\node (v1) at (0,-0.5) [circle, draw = black, fill = black, inner sep = 0.5mm, label=left:{$p_1$}] {};
\node (v2) at (0,0.5) [circle, draw = black, fill = black, inner sep = 0.5mm, label=left:{$p_2$}] {};
\draw (v1) -- (v2);
\node at (0,-1.5) {$\mathcal{P}_1$};
\end{tikzpicture}\begin{tikzpicture}[scale=0.65]
	\node (1) at (0, 0) [circle, draw = black, fill = black, inner sep = 0.5mm, label=left:{$p_1$}]{};
	\node (2) at (0, 1)[circle, draw = black, fill = black, inner sep = 0.5mm, label=left:{$p_2$}] {};
	\node (3) at (-0.5, 2) [circle, draw = black, fill = black, inner sep = 0.5mm, label=left:{$p_3$}] {};
	\node (4) at (0.5, 2) [circle, draw = black, fill = black, inner sep = 0.5mm, label=right:{$p_4$}] {};
    \draw (1)--(2);
    \draw (2)--(3);
    \draw (2)--(4);
    \node at (0,-1) {$\mathcal{P}_2$};
\end{tikzpicture}
\begin{tikzpicture}[scale=0.65]
\node (v1) at (1.5,0) [circle, draw = black, fill = black, inner sep = 0.5mm, label=left:{$p_1$}] {};
\node (v4) at (2.5,0) [circle, draw = black, fill = black, inner sep = 0.5mm, label=right:{$p_2$}] {};
\node (v2) at (2,1) [circle, draw = black, fill = black, inner sep = 0.5mm, label=right:{$p_3$}] {};
\node (v3) at (2,2) [circle, draw = black, fill = black, inner sep = 0.5mm, label=right:{$p_4$}] {};
\draw (v1) -- (v2) -- (v3);
\draw (v2) -- (v4);
\node at (2,-1) {$\mathcal{P}^*_2$};
\end{tikzpicture}\begin{tikzpicture}[scale=0.65]
\node (v4) at (-2,3) [circle, draw = black, fill = black, inner sep = 0.5mm, label=left:{$p_5$}] {};
\node (v6) at (-0.5,3) [circle, draw = black, fill = black, inner sep = 0.5mm, label=right:{$p_6$}] {};
\node (v5) at (-0.5,2) [circle, draw = black, fill = black, inner sep = 0.5mm, label=right:{$p_4$}] {};
\node (v3) at (-2,2) [circle, draw = black, fill = black, inner sep = 0.5mm, label=left:{$p_3$}] {};
\node (v2) at (-1.25,1) [circle, draw = black, fill = black, inner sep = 0.5mm, label=left:{$p_2$}] {};
\node (v1) at (-1.25,0) [circle, draw = black, fill = black, inner sep = 0.5mm, label=left:{$p_1$}] {};
\draw (v1) -- (v2) -- (v3) -- (v4);
\draw (v2) -- (v5) -- (v6);
\node at (-1.25,-1) {$\mathcal{P}_3$};
\end{tikzpicture}
\begin{tikzpicture}[scale=0.65]
\node (v7) at (1,0)  [circle, draw = black, fill = black, inner sep = 0.5mm, label=left:{$p_1$}] {};
\node (v11) at (2.5,0)  [circle, draw = black, fill = black, inner sep = 0.5mm, label=right:{$p_2$}] {};
\node (v12) at (2.5,1)  [circle, draw = black, fill = black, inner sep = 0.5mm, label=right:{$p_4$}] {};
\node (v8) at (1,1)  [circle, draw = black, fill = black, inner sep = 0.5mm, label=left:{$p_3$}] {};
\node (v9) at (1.75,2)  [circle, draw = black, fill = black, inner sep = 0.5mm, label=right:{$p_5$}] {};
\node (v10) at (1.75,3)  [circle, draw = black, fill = black, inner sep = 0.5mm, label=right:{$p_6$}] {};
\draw (v7) -- (v8) -- (v9) -- (v10);
\draw (v11) -- (v12) -- (v9);
\node at (1.75,-1) {$\mathcal{P}^*_3$};
\end{tikzpicture}$$ $$\begin{tikzpicture}[scale=0.8]
\node [circle, draw = black, fill = black, inner sep = 0.5mm, label=right:{$p_1$}] (v13) at (9,0.5) {};
\node [circle, draw = black, fill = black, inner sep = 0.5mm, label=right:{$p_{\lfloor\frac{n}{2}\rfloor-1}$}] (v15) at (9,1.25) {};
\node [circle, draw = black, fill = black, inner sep = 0.5mm, label=right:{$p_{\lfloor\frac{n}{2}\rfloor}$}] (v19) at (9,2) {};
\node [circle, draw = black, fill = black, inner sep = 0.5mm, label=right:{$p_{\lfloor\frac{n}{2}\rfloor+1}$}] (v16) at (9,2.75) {};
\node [circle, draw = black, fill = black, inner sep = 0.5mm, label=right:{$p_{n-1}$}] (v18) at (9,3.5) {};
\draw (v15) -- (v16) -- cycle;
\node [circle, draw = black, fill = black, inner sep = 0.5mm, label=left:{$p_n$}] (v20) at (8.5,2.75) {};
\draw (v19) -- (v20);
\node at (9,3.25) {$\vdots$};
\node at (9,1) {$\vdots$};
\node at (9,-0.5) {$\mathcal{P}_{4,n}$};
\end{tikzpicture}\quad\begin{tikzpicture}[scale=0.8]
\node [circle, draw = black, fill = black, inner sep = 0.5mm, label=right:{$p_1$}] (v13) at (11,0.5) {};
\node [circle, draw = black, fill = black, inner sep = 0.5mm, label=right:{$p_{\lfloor\frac{n-1}{2}\rfloor}$}] (v15) at (11,1.25) {};
\node [circle, draw = black, fill = black, inner sep = 0.5mm, label=right:{$p_{\lfloor\frac{n-1}{2}\rfloor+2}$}] (v19) at (11,2) {};
\node [circle, draw = black, fill = black, inner sep = 0.5mm, label=right:{$p_{\lfloor\frac{n-1}{2}\rfloor+3}$}] (v16) at (11,2.75) {};
\node [circle, draw = black, fill = black, inner sep = 0.5mm, label=right:{$p_{n}$}] (v18) at (11,3.5) {};
\draw (v15) -- (v16) -- cycle;
\node [circle, draw = black, fill = black, inner sep = 0.5mm, label=left:{$p_{\lfloor\frac{n-1}{2}\rfloor+1}$}] (v20) at (10.5,1.25) {};
\draw (v19) -- (v20);
\node at (11,3.25) {$\vdots$};
\node at (11,1) {$\vdots$};
\node at (11,-0.5) {$\mathcal{P}^*_{4,n}$};
\end{tikzpicture}\quad
\begin{tikzpicture}
	\node (1) at (0, 0) [circle, draw = black, fill = black, inner sep = 0.5mm, label=left:{$p_1$}]{};
	\node (2) at (-0.5, 0.5)[circle, draw = black, fill = black, inner sep = 0.5mm, label=left:{$p_2$}] {};
	\node (3) at (0.5, 0.5) [circle, draw = black, fill = black, inner sep = 0.5mm, label=right:{$p_3$}] {};
    \node (4) at (-0.5, 1) [circle, draw = black, fill = black, inner sep = 0.5mm, label=left:{$p_4$}] {};
    \node (5) at (0.5, 1) [circle, draw = black, fill = black, inner sep = 0.5mm, label=right:{$p_5$}] {};
    \node (6) at (0, 1.5) {$\vdots$};
    \node (7) at (-0.5, 2)[circle, draw = black, fill = black, inner sep = 0.5mm, label=left:{$p_{2n-2}$}] {};
	\node (8) at (0.5, 2) [circle, draw = black, fill = black, inner sep = 0.5mm, label=right:{$p_{2n-1}$}] {};
    \node (9) at (-0.5, 2.5) [circle, draw = black, fill = black, inner sep = 0.5mm, label=left:{$p_{2n}$}] {};
    \node (10) at (0.5, 2.5) [circle, draw = black, fill = black, inner sep = 0.5mm, label=right:{$p_{2n+1}$}] {};
    \node at (0,-0.75) {$\mathcal{P}_{5,n}$};
    \draw (4)--(3)--(1)--(2)--(5);
    \draw (5)--(3);
    \draw (2)--(4);
    \draw (7)--(9)--(8);
    \draw (7)--(10)--(8);
\end{tikzpicture}\quad\begin{tikzpicture}
	\node (1) at (0, 2.5) [circle, draw = black, fill = black, inner sep = 0.5mm, label=left:{$p_{2n+1}$}]{};
	\node (2) at (-0.5, 2)[circle, draw = black, fill = black, inner sep = 0.5mm, label=left:{$p_{2n-1}$}] {};
	\node (3) at (0.5, 2) [circle, draw = black, fill = black, inner sep = 0.5mm, label=right:{$p_{2n}$}] {};
    \node (4) at (-0.5, 1.5) [circle, draw = black, fill = black, inner sep = 0.5mm, label=left:{$p_{2n-3}$}] {};
    \node (5) at (0.5, 1.5) [circle, draw = black, fill = black, inner sep = 0.5mm, label=right:{$p_{2n-2}$}] {};
    \node (6) at (0, 1) {$\vdots$};
    \node (7) at (-0.5, 0.5)[circle, draw = black, fill = black, inner sep = 0.5mm, label=left:{$p_{3}$}] {};
	\node (8) at (0.5, 0.5) [circle, draw = black, fill = black, inner sep = 0.5mm, label=right:{$p_{4}$}] {};
    \node (9) at (-0.5, 0) [circle, draw = black, fill = black, inner sep = 0.5mm, label=left:{$p_{1}$}] {};
    \node (10) at (0.5, 0) [circle, draw = black, fill = black, inner sep = 0.5mm, label=right:{$p_{2}$}] {};
    \node at (0,-0.75) {$\mathcal{P}^*_{5,n}$};
    \draw (4)--(3)--(1)--(2)--(5);
    \draw (5)--(3);
    \draw (2)--(4);
    \draw (7)--(9)--(8);
    \draw (7)--(10)--(8);
\end{tikzpicture}$$
\caption{Posets of toral-pairs}\label{fig:bb}
\end{figure}
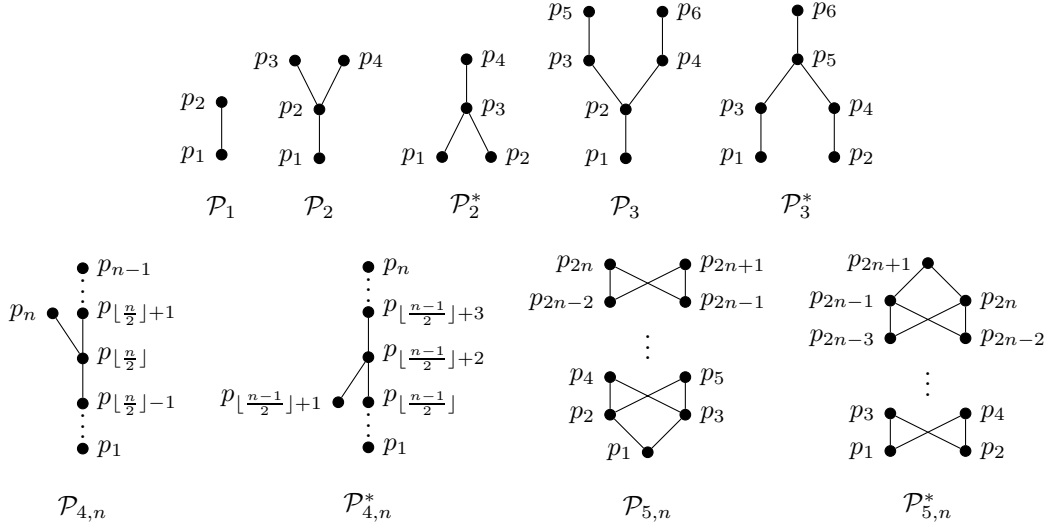
\end{Ex}

\begin{remark}\label{rem:proofheur}  
We can show that a functional $F\in\mathfrak{g}^*$ is Frobenius in the following way. Let $\{x_1,\hdots,x_n\}$ be a vector space basis for $\mathfrak{g}$ and let $B=\sum_{i=1}^nb_jx_j\in\mathfrak{g}\cap \ker(F)$. Now, determine the restrictions $F([x_i,B])=0$ places on the $b_j$, for $i=1,\hdots,n$ and $j=1,\hdots,n$. Finally, show that these restrictions imply that  $B$ is the zero matrix; that is, $$\dim\ker(F)=0\ge\ind(\mathfrak{g})\ge 0.$$ 
\end{remark}

\begin{remark}
Given a poset $\mathcal{P}=\{1,\hdots,n\}$ and a functional $F\in(\mathfrak{g}_A(\mathcal{P}))^*$, it is shown in Appendix A that to determine $\dim\ker(F)$ it suffices to consider the following system of equations:
\begin{itemize}
    \item $F([E_{p_i,p_i},B])=0$, for $p_i\in\mathcal{P}$, and
    \item $F([E_{p_i,p_j},B])=0$, for $p_i,p_j\in\mathcal{P}$ satisfying $p_i\preceq p_j$.
\end{itemize}
\end{remark}

\begin{theorem}\label{thm:h1bb}
If $\mathcal{P}_1=\{p_1,p_2\}$ with $p_1\preceq p_2$ and $F_{\mathcal{P}_1}=E^*_{p_1,p_2}$, then $(\mathcal{P}_1,F_{\mathcal{P}_1})$ forms a toral-pair.
\end{theorem}
\begin{proof}
To ease notation, let $\mathcal{P}=\mathcal{P}_1$ and $F=F_{\mathcal{P}_1}$. It is clear that $|Ext(\mathcal{P})|=2$ and $\Sigma(\mathcal{P})$ is contractible so that \textbf{(P1)} and \textbf{(P3)} of Definition~\ref{def:toralpair} are satisfied. Now, taking $B\in \mathfrak{g}_A(\mathcal{P})\cap\ker(F)$ we have
\begin{itemize}
    \item $F([E_{p_1,p_2},B])=E^*_{p_2,p_2}(B)-E^*_{p_1,p_1}(B)=0$,
    \item $F([E_{p_1,p_1},B])=E^*_{p_1,p_2}(B)=0$.
\end{itemize}
Thus, $$E^*_{p_1,p_1}(B)=E^*_{p_2,p_2}(B)$$ and $$E^*_{p_1,p_2}(B)=0$$ so that $F$ satisfies \textbf{(F4)} of Definition~\ref{def:toralpair}. Furthermore, since $B\in\mathfrak{sl}(2)$, i.e., $E^*_{p_1,p_1}(B)+E^*_{p_2,p_2}(B)=0$, we conclude that $$E^*_{p_1,p_1}(B)=E^*_{p_2,p_2}(B)=0;$$ that is, $B=0$ and we have shown that $\mathfrak{g}_A(\mathcal{P})$ is Frobenius with Frobenius functional $F=E^*_{p_1,p_2}$.

Given the form of the Frobenius functional $F$, we have that
$F$ satisfies \textbf{(F1)} through \textbf{(F4)} of Definition~\ref{def:toralpair} as follows:
\begin{itemize}
    \item $F$ is clearly small,
    \item $D_F(\mathcal{P})=\{p_1\}$ forms an ideal of $\mathcal{P}$, $U_F(\mathcal{P})=\{p_2\}$ forms a filter, and $B_F(\mathcal{P})=\emptyset$,
    \item $\Gamma_F$ contains the only edge $(p_1,p_2)$ between elements of $Ext(\mathcal{P}_1)$, and
    \item \textbf{(F4)} was established above.
\end{itemize}
It remains to show that \textbf{(P2)} is satisfied; that is, $\mathfrak{g}_A(\mathcal{P})$ has a spectrum consisting of an equal number of 0's and 1's. To determine the spectrum of $\mathfrak{g}_A(\mathcal{P})$, it suffices to calculate $[\widehat{F},x]$ for $x\in \mathscr{B}_{\mathcal{P},F}$; but $\mathscr{B}_{\mathcal{P},F}$ has only two elements: $E_{p_1,p_1}-E_{p_2,p_2}$ and $E_{p_1,p_2}$. The former is an eigenvector of $ad(\widehat{F})$ with eigenvalue 0, and the latter is an eigenvector with eigenvalue 1. Therefore, $\mathfrak{g}_A(\mathcal{P})$ has a binary spectrum and $(\mathcal{P},F)$ forms a toral-pair.
\end{proof}

\begin{theorem}\label{thm:h2bb}
Each of the following pairs, consisting of a poset $\mathcal{P}$ and a functional $F_{\mathcal{P}}$, form a toral-pair $(\mathcal{P},F_{\mathcal{P}})$.
\begin{enumerate} [label=\textup(\roman*\textup)]
        \item $\mathcal{P}_2=\{p_1,p_2,p_3,p_4\}$ with $p_1\preceq p_2\preceq p_3,p_4$, and $$F_{\mathcal{P}_2}=E^*_{p_1,p_3}+E^*_{p_1,p_4}+E^*_{p_2,p_4},$$
        \item $\mathcal{P}^*_2=\{p_1,p_2,p_3,p_4\}$ with $p_1,p_2\preceq p_3\preceq p_4$, and $$F_{\mathcal{P}^*_2}=E^*_{p_1,p_4}+E^*_{p_2,p_4}+E^*_{p_2,p_3}.$$
\end{enumerate}
\end{theorem}
\begin{proof}
We prove (i), as (ii) follows via a symmetric argument. To ease notation, let $\mathcal{P}=\mathcal{P}_{2}$ and $F=F_{\mathcal{P}_{2}}$. It is clear that $|Ext(\mathcal{P})|=3$ and $\Sigma(\mathcal{P})$ is contractible so that \textbf{(P1)} and \textbf{(P3)} of Definition~\ref{def:toralpair} are satisfied. Now, assume $B\in\mathfrak{g}_A(\mathcal{P})\cap\ker(F)$. Then $B$ must satisfy the following restrictions, which are broken into 3 groups:

\bigskip
\noindent
\textbf{Group 1:}
\begin{itemize}
    \item $F([E_{p_2,p_2},B])=E_{p_2,p_3}^*(B)=0$,
    \item $F([E_{p_4,p_4},B])=-E_{p_1,p_4}^*(B)=0$,
    \item $F([E_{p_2,p_3},B])=-E^*_{p_1,p_2}(B)=0$.
\end{itemize}
\textbf{Group 2:}
\begin{itemize}
    \item $F([E_{p_1,p_1},B])=E_{p_1,p_2}^*(B)+E_{p_1,p_4}^*(B)=0$,
    \item $F([E_{p_3,p_3},B])=-E_{p_1,p_3}^*(B)-E_{p_2,p_3}^*(B)=0$,
    \item $F([E_{p_1,p_2},B])=E^*_{p_2,p_3}(B)+E^*_{p_2,p_4}(B)=0$,
\end{itemize}
\textbf{Group 3:}
\begin{itemize}
    \item $F([E_{p_1,p_3},B])=E^*_{p_3,p_3}(B)-E^*_{p_1,p_1}(B)=0$,
    \item $F([E_{p_1,p_4},B])=E^*_{p_4,p_4}(B)-E^*_{p_1,p_1}(B)=0$,
    \item $F([E_{p_2,p_4},B])=E^*_{p_4,p_4}(B)-E^*_{p_2,p_2}(B)-E^*_{p_1,p_2}(B)=0$.
\end{itemize}
The restrictions of the equations in Group 1 immediately imply that
\begin{equation}\label{eqn:h2bb1}
E_{p_2,p_3}^*(B)=E_{p_1,p_4}^*(B)=E^*_{p_1,p_2}(B)=0.
\end{equation} 
Combining the Group 1 restrictions to those of Group 2, we may conclude that
\begin{equation}\label{eqn:h2bb2}
E_{p_1,p_3}^*(B)=E^*_{p_2,p_4}(B)=0.
\end{equation} 
Finally, combining the restrictions of Group 1 to those of Group 3, we find that
\begin{equation}\label{eqn:h2bb3}
E^*_{p_i,p_i}(B)=E^*_{p_j,p_j}(B)\text{ for all }p_i,p_j\in\mathcal{P}.
\end{equation}
Since $B\in \mathfrak{sl}(4)$, equation (\ref{eqn:h2bb3}) implies that $E^*_{p_i,p_i}(B)=0$ for all $p_i\in\mathcal{P}$. Thus, $B=0$ and $\mathfrak{g}_A(\mathcal{P})$ is Frobenius with a choice of Frobenius functional given by $F$.

Given the form of the Frobenius functional $F$, we have that
$F$ satisfies \textbf{(F1)} through \textbf{(F4)} of Definition~\ref{def:toralpair} as follows:
\begin{itemize}
    \item $F$ is clearly small,
    \item $D_F(\mathcal{P})=\{p_1,p_2\}$ forms an ideal of $\mathcal{P}$, $U_F(\mathcal{P})=\{p_3,p_4\}$ forms a filter, and $B_F(\mathcal{P})=\emptyset$;
    \item $\Gamma_F$ contains the only edges, $(p_1,p_3)$ and $(p_1,p_4)$, between elements of $Ext(\mathcal{P})$; and
    \item equations (\ref{eqn:h2bb1}), (\ref{eqn:h2bb2}), and (\ref{eqn:h2bb3}) above show that $F$ satisfies \textbf{(F4)} of Definition~\ref{def:toralpair}.
\end{itemize}
It remains to show that \textbf{(P2)} is satisfied; that is, $\mathfrak{g}_A(\mathcal{P})$ has a spectrum consisting of an equal number of 0's and 1's. To determine the spectrum of $\mathfrak{g}_A(\mathcal{P})$, it suffices to calculate $[\widehat{F},x]$ for $x\in \mathscr{B}_{\mathcal{P},F}$. Note that $\mathscr{B}_{\mathcal{P},F}$ can be partitioned into two sets:
$$G_0=\{E_{p_3,p_3}-E_{p_1,p_1},E_{p_4,p_4}-E_{p_1,p_1},E_{p_4,p_4}-E_{p_2,p_2},E_{p_1,p_2}\},$$
which consists of eigenvectors of $ad(\widehat{F})$ with eigenvalue 0, and
$$G_1=\{E_{p_1,p_3},E_{p_1,p_4},E_{p_2,p_3},E_{p_2,p_4}\}$$
 which consists of eigenvectors with eigenvalue 1. As $|G_0|=|G_1|$, we conclude that $\mathfrak{g}_A(\mathcal{P})$ has a binary spectrum and $(\mathcal{P},F)$ forms a toral-pair.
\end{proof}

For the remaining posets, we relegate the proofs that the poset along with the corresponding Frobenius functional form a toral-pair to Appendix B.

\begin{theorem}\label{thm:h3bb}
Each of the following pairs, consisting of a poset $\mathcal{P}$ and a functional $F_{\mathcal{P}}$, form a toral-pair $(\mathcal{P},F_{\mathcal{P}})$.
\begin{enumerate} [label=\textup(\roman*\textup)]
        \item $\mathcal{P}_3=\{p_1,p_2,p_3,p_4,p_5,p_6\}$ with $p_1\preceq p_2\preceq p_3,p_4$; $p_3\preceq p_5$; and $p_4\preceq p_6$, and $$F_{\mathcal{P}_3}=E^*_{p_1,p_5}+E^*_{p_1,p_6}+E^*_{p_2,p_3}+E^*_{p_2,p_4}+E^*_{p_2,p_6},$$
        \item $\mathcal{P}^*_3=\{p_1,p_2,p_3,p_4,p_5,p_6\}$ with $p_1\preceq p_3$; $p_2\preceq p_4$; and $p_3,p_4\preceq p_5\preceq p_6$, and $$F_{\mathcal{P}_3^*}=E^*_{p_1,p_6}+E^*_{p_2,p_6}+E^*_{p_3,p_5}+E^*_{p_4,p_5}+E^*_{p_2,p_5},$$
        \item $\mathcal{P}_{4,n}=\{p_1,\hdots,p_n\}$ with $p_1\preceq p_2\preceq\hdots\preceq p_{n-1}$ as well as $p_1\preceq p_2\preceq\hdots\preceq p_{\lfloor\frac{n}{2}\rfloor}\preceq p_n$, and $$F_{\mathcal{P}_{4,n}}=\sum_{i=1}^{\lfloor\frac{n-1}{2}\rfloor}E^*_{p_i,p_{n-i}}+\sum_{i=1}^{\lfloor\frac{n}{2}\rfloor}E^*_{p_i,p_n},$$
        \item $\mathcal{P}^*_{4,n}=\{p_1,\hdots,p_n\}$ with $p_1\preceq p_2\preceq \hdots\preceq p_{\lfloor\frac{n-1}{2}\rfloor}\preceq p_{\lfloor\frac{n-1}{2}\rfloor+2}\preceq\hdots\preceq p_{n}$ as well as $p_{\lfloor\frac{n-1}{2}\rfloor+1}\preceq p_{\lfloor\frac{n-1}{2}\rfloor+2}\preceq \hdots\preceq p_{n}$, and $$F_{\mathcal{P}^*_{4,n}}=\sum_{i=1}^{\lfloor\frac{n-1}{2}\rfloor}E^*_{p_i,p_{n+1-i}}+\sum_{i=\lceil\frac{n}{2}\rceil}^{n}E^*_{p_{\lceil\frac{n}{2}\rceil},p_i},$$
        \item $\mathcal{P}_{5,n}=\{p_1,\hdots,p_{2n+1}\}$ with $p_i\preceq p_j$ for $1\le i<2n$ odd and $i+1\le j\le 2n+1$ as well as $p_i\preceq p_j$ for $1<i<2n$ even and $i+2\le j\le 2n+1$, and $$F_{\mathcal{P}_{5,n}}=E^*_{p_1,p_{2n+1}}+\sum_{i=1}^{\lceil\frac{n-1}{2}\rceil}E^*_{p_i, p_{2n}}+\sum_{k=1}^{\lfloor\frac{n-1}{2}\rfloor}E^*_{p_{2k}, p_{2n-2k}}+\sum_{k=1}^{\lfloor\frac{n-1}{2}\rfloor}E^*_{p_{2k+1}, p_{2n-2k+1}},$$
        \item $\mathcal{P}^*_{5,n}=\{p_1,\hdots,p_{2n+1}\}$ with $p_i\preceq p_j$ for $1\le i<2n$ odd and $i+2\le j\le 2n+1$, as well as $p_i\preceq p_j$ for $1< i<2n$ even and $i+1\le j\le 2n+1$, and $$F_{\mathcal{P}^*_{5,n}}=E^*_{p_1,p_{2n+1}}+\sum_{i=2\lfloor\frac{n+1}{4}\rfloor+1}^{2n+1}E^*_{p_2, p_i}+\sum_{k=2}^{\lfloor\frac{n+1}{2}\rfloor}E^*_{p_{2k}, p_{2n-2k+4}}+\sum_{k=1}^{\lfloor\frac{n-1}{2}\rfloor}E^*_{p_{2k+1}, p_{2n-2k+1}},$$
\end{enumerate}
\end{theorem}
\begin{proof}
Appendix B.
\end{proof}

\noindent
Thus, the posets $\mathcal{P}_1$, $\mathcal{P}_2$, $\mathcal{P}^*_2$, $\mathcal{P}_3$, $\mathcal{P}^*_3$, $\mathcal{P}_{4,n}$, $\mathcal{P}^*_{4,n}$ $\mathcal{P}_{5,n}$, and $\mathcal{P}^*_{5,n}$ along with the corresponding Frobenius functionals found in Theorems~\ref{thm:h1bb}-\ref{thm:h3bb} form toral-pairs. In the next section, posets of toral-pairs are combined to form toral posets.

\section{Toral posets}\label{sec:toralpos}
In this section, we define toral posets, which are constructed inductively from the posets of toral-pairs. Furthermore, we show that if $\mathcal{P}$ is a toral poset, then $\Sigma(\mathcal{P})$ is homotopic to a wedge sum of $\ind\mathfrak{g}_A(\mathcal{P})$ one-spheres.

Let $(\mathcal{S}, F)$ be a toral-pair and $\mathcal{Q}$
be a poset. We define twelve ways of ``combining" the posets $\mathcal{S}$ and $\mathcal{Q}$ by identifying minimal (resp., maximal) elements of $\mathcal{S}$ with minimal (resp., maximal) elements of $\mathcal{Q}$. If $|Ext(\mathcal{S})|=2$, then $Ext(\mathcal{S})=\{a,b\}$ with either $a\preceq_{\mathcal{S}}b$ or $b\preceq_{\mathcal{S}}a$; and if $|Ext(\mathcal{S})|=3$, then $Ext(\mathcal{S})=\{a,b,c\}$ with either $a\preceq_{\mathcal{S}} b,c$ or $b,c\preceq_{\mathcal{S}} a$. Further, assume $x,y,z\in Ext(\mathcal{Q})$. Since the construction rules are defined by identifying minimal elements and maximal elements of $\mathcal{S}$ and $\mathcal{Q}$, assume that if $a,b,$ or $c$ are identified with elements of $\mathcal{Q}$, then those elements are $x,y,$ or $z$, respectively. To ease notation, let $\sim_{\mathcal{P}}$ denote that two elements of a poset $\mathcal{P}$ are related, and let $\nsim_{\mathcal{P}}$ denote that two elements are not related; that is, for $x,y\in\mathcal{P}$, $x\sim_{\mathcal{P}} y$ denotes that $x\preceq_{\mathcal{P}}y$ or $y\preceq_{\mathcal{P}}x$, and $x\nsim_{\mathcal{P}} y$ denotes that both $x\npreceq_{\mathcal{P}}y$ and $y\npreceq_{\mathcal{P}}x$. The construction rules are as follows: If $|Ext(\mathcal{S})|=2$ or $3$, then
\begin{itemize}
    \item $\textup{A}_1$ denotes identifying $b\in Ext(\mathcal{S})$ with $y\in Ext(\mathcal{Q})$,
    \item $\textup{C}$ denotes identifying $a\in Ext(\mathcal{S})$ with $x\in Ext(\mathcal{Q})$,
    \item $\textup{E}_1$ denotes identifying $a,b\in Ext(\mathcal{S})$ with $x,y\in Ext(\mathcal{Q})$, where $x\nsim_{\mathcal{Q}} y$.
\end{itemize}
If $Ext(\mathcal{S})=3$, then
\begin{itemize}
    \item $\textup{A}_2$ denotes identifying $c\in Ext(\mathcal{S})$ with $z\in Ext(\mathcal{Q})$,
    \item $\textup{B}$ denotes identifying $b,c\in Ext(\mathcal{S})$ with $y,z\in Ext(\mathcal{Q})$,
    \item $\textup{D}_1$ denotes identifying $a,b\in Ext(\mathcal{S})$ with $x,y\in Ext(\mathcal{Q})$, where $x\sim_{\mathcal{Q}} y$,
    \item $\textup{D}_2$ denotes identifying $a,c\in Ext(\mathcal{S})$ with $x,z\in Ext(\mathcal{Q})$, where $x\sim_{\mathcal{Q}} z$,
    \item $E_2$ denotes identifying $a,c\in Ext(\mathcal{S})$ with $x,z\in Ext(\mathcal{Q})$, where $x\nsim_{\mathcal{Q}} z$,
    \item $\textup{F}$ denotes identifying $a,b,c\in Ext(\mathcal{S})$ with $x,y,z\in Ext(\mathcal{Q})$, where $x\sim_{\mathcal{Q}} y$ and $x\sim_{\mathcal{Q}} z$,
    \item $\textup{G}_1$ denotes identifying $a,b,c\in Ext(\mathcal{S})$ with $x,y,z\in Ext(\mathcal{Q})$, where $x\sim_{\mathcal{Q}} y$ and $x\nsim_{\mathcal{Q}} z$,
    \item $\textup{G}_2$ denotes identifying $a,b,c\in Ext(\mathcal{S})$ with $x,y,z\in Ext(\mathcal{Q})$, where $x\nsim_{\mathcal{Q}} y$ and $x\sim_{\mathcal{Q}} z$,
    \item $\textup{H}$ denotes identifying $a,b,c\in Ext(\mathcal{S})$ with $x,y,z\in Ext(\mathcal{Q})$, where $x\nsim_{\mathcal{Q}} y$ and $x\nsim_{\mathcal{Q}} z$.
\end{itemize}

\begin{definition}\label{def:toral}
A poset $\mathcal{P}$ is called \textit{toral} if there exists a sequence of toral-pairs $\{(\mathcal{S}_i,F_i)\}_{i=1}^n$ along with a sequence of posets $\mathcal{S}_1=\mathcal{Q}_1\subset\mathcal{Q}_2\subset\hdots\subset\mathcal{Q}_n=\mathcal{P}$ such that $\mathcal{Q}_{i}$ is formed from $\mathcal{Q}_{i-1}$ and $\mathcal{S}_{i}$ by applying a rule from the set $\{A_1,A_2,B,C, D_1,D_2,E_1,E_2,F,G_1,G_2,H\}$, for $i=2,\hdots,n$. Such a sequence
$\mathcal{S}_1=\mathcal{Q}_1\subset\mathcal{Q}_2\subset\hdots\subset\mathcal{Q}_n=\mathcal{P}$ is called a construction sequence for $\mathcal{P}$.
\end{definition}

\begin{Ex} Let $\mathcal{P}$ be the toral poset constructed from the toral-pairs $\{(\mathcal{S}_i,F_i)\}_{i=1}^5$, where $\mathcal{S}_i=\mathcal{P}_2$, for $i=1,\hdots,5$, with attendant construction sequence $\mathcal{S}_1=\mathcal{Q}_1\subset \mathcal{Q}_2\subset\mathcal{Q}_3\subset\mathcal{Q}_4\subset\mathcal{Q}_5=\mathcal{P}$, where: $\mathcal{Q}_2$ is formed from $\mathcal{Q}_1$ and $\mathcal{S}_2$ by applying rule $A_1$, $\mathcal{Q}_3$ is formed from $\mathcal{Q}_2$ and $\mathcal{S}_3$ by applying rule $C$, $\mathcal{Q}_4$ is formed from $\mathcal{Q}_3$ and $\mathcal{S}_4$ by applying rule $D_1$, and $\mathcal{Q}_5=\mathcal{P}$ is formed from $\mathcal{Q}_4$ and $\mathcal{S}_5$ by applying rule $F$. See Figure~\ref{fig:toral}.
\begin{figure}[H]
$$\begin{tikzpicture}[scale=0.4]
\node [circle, draw = black, fill = black, inner sep = 0.5mm] (v1) at (-2.5,0.5) {};
\node [circle, draw = black, fill = black, inner sep = 0.5mm] (v2) at (-2.5,1.5) {};
\node [circle, draw = black, fill = black, inner sep = 0.5mm] (v4) at (-3,2.5) {};
\node [circle, draw = black, fill = black, inner sep = 0.5mm] (v3) at (-2,2.5) {};
\draw (v1) -- (v2) -- (v3);
\draw (v2) -- (v4);
\draw[->] (-1.5,1.5) -- (-0.5,1.5);
\node [circle, draw = black, fill = black, inner sep = 0.5mm] (v7) at (0,2.5) {};
\node [circle, draw = black, fill = black, inner sep = 0.5mm] (v8) at (1,2.5) {};
\node [circle, draw = black, fill = black, inner sep = 0.5mm] (v11) at (2,2.5) {};
\node [circle, draw = black, fill = black, inner sep = 0.5mm] (v10) at (1.5,1.5) {};
\node [circle, draw = black, fill = black, inner sep = 0.5mm] (v6) at (0.5,1.5) {};
\node [circle, draw = black, fill = black, inner sep = 0.5mm] (v5) at (0.5,0.5) {};
\node [circle, draw = black, fill = black, inner sep = 0.5mm] (v9) at (1.5,0.5) {};
\node at (-1,2) {$A_1$};
\draw (v5) -- (v6) -- (v7);
\draw (v6) -- (v8);
\draw (v9) -- (v10) -- (v8);
\draw (v10) -- (v11);
\draw [->] (2.5,1.5) -- (3.5,1.5);
\node at (3,2) {$C$};
\node [circle, draw = black, fill = black, inner sep = 0.5mm] (v14) at (4,2.5) {};
\node [circle, draw = black, fill = black, inner sep = 0.5mm] (v15) at (5,2.5) {};
\node [circle, draw = black, fill = black, inner sep = 0.5mm] (v18) at (6,2.5) {};
\node [circle, draw = black, fill = black, inner sep = 0.5mm] (v20) at (7,2.5) {};
\node [circle, draw = black, fill = black, inner sep = 0.5mm] (v21) at (8,2.5) {};
\node [circle, draw = black, fill = black, inner sep = 0.5mm] (v13) at (4.5,1.5) {};
\node [circle, draw = black, fill = black, inner sep = 0.5mm] (v17) at (5.5,1.5) {};
\node [circle, draw = black, fill = black, inner sep = 0.5mm] (v19) at (6.5,1.5) {};
\node [circle, draw = black, fill = black, inner sep = 0.5mm] (v12) at (4.5,0.5) {};
\node [circle, draw = black, fill = black, inner sep = 0.5mm] (v16) at (5.5,0.5) {};
\draw (v12) -- (v13) -- (v14);
\draw (v13) -- (v15);
\draw (v16) -- (v17) -- (v15);
\draw (v17) -- (v18);
\draw (v16) -- (v19) -- (v20);
\draw (v19) -- (v21);
\draw[->] (8.5,1.5) -- (9.5,1.5);
\node at (9,2) {$D_1$};
\node  (v24) at (10,2.5) [circle, draw = black, fill = black, inner sep = 0.5mm] {};
\node  (v25) at (11,2.5) [circle, draw = black, fill = black, inner sep = 0.5mm] {};
\node (v28) at (12,2.5) [circle, draw = black, fill = black, inner sep = 0.5mm] {};
\node (v31) at (13,2.5) [circle, draw = black, fill = black, inner sep = 0.5mm] {};
\node (v32) at (14,2.5) [circle, draw = black, fill = black, inner sep = 0.5mm] {};
\node  (v23) at (10.5,1.5) [circle, draw = black, fill = black, inner sep = 0.5mm] {};
\node  (v27) at (11.5,1.5) [circle, draw = black, fill = black, inner sep = 0.5mm] {};
\node  (v29) at (13.5,1.5) [circle, draw = black, fill = black, inner sep = 0.5mm] {};
\node  (v22) at (10.5,0.5) [circle, draw = black, fill = black, inner sep = 0.5mm] {};
\node  (v26) at (11.5,0.5) [circle, draw = black, fill = black, inner sep = 0.5mm] {};
\node  (v30) at (12.5,1.5) [circle, draw = black, fill = black, inner sep = 0.5mm] {};
\draw (v22) -- (v23) -- (v24);
\draw (v23) -- (v25);
\draw (v26) -- (v27) -- (v25);
\draw (v27) -- (v28);
\draw (v26) -- (v29) -- (v28);
\draw (v26) -- (v30);
\draw (v31) -- (v30) -- (v32);
\node [circle, draw = black, fill = black, inner sep = 0.5mm] (v33) at (15,2.5) {};
\draw  (v29) -- (v33);
\draw [->] (15.5,1.5) -- (16.5,1.5);
\node at (16,2) {$F$};
\node [circle, draw = black, fill = black, inner sep = 0.5mm] (v46) at (17,1.5) {};
\node [circle, draw = black, fill = black, inner sep = 0.5mm] (v36) at (17.5,2.5) {};
\node [circle, draw = black, fill = black, inner sep = 0.5mm] (v37) at (18.5,2.5) {};
\node [circle, draw = black, fill = black, inner sep = 0.5mm] (v40) at (19.5,2.5) {};
\node [circle, draw = black, fill = black, inner sep = 0.5mm] (v42) at (20.5,2.5) {};
\node [circle, draw = black, fill = black, inner sep = 0.5mm] (v43) at (21.5,2.5) {};
\node [circle, draw = black, fill = black, inner sep = 0.5mm] (v45) at (22.5,2.5) {};
\node [circle, draw = black, fill = black, inner sep = 0.5mm] (v35) at (18,1.5) {};
\node [circle, draw = black, fill = black, inner sep = 0.5mm] (v39) at (19,1.5) {};
\node [circle, draw = black, fill = black, inner sep = 0.5mm] (v41) at (20,1.5) {};
\node [circle, draw = black, fill = black, inner sep = 0.5mm] (v44) at (21,1.5) {};
\node [circle, draw = black, fill = black, inner sep = 0.5mm] (v34) at (18,0.5) {};
\node [circle, draw = black, fill = black, inner sep = 0.5mm] (v38) at (19,0.5) {};
\draw (v34) -- (v35) -- (v36);
\draw (v35) -- (v37);
\draw (v38) -- (v39) -- (v37);
\draw (v39) -- (v40);
\draw (v38) -- (v41) -- (v42);
\draw (v41) -- (v43);
\draw (v38) -- (v44) -- (v40);
\draw (v44) -- (v45);
\draw (v34) -- (v46) -- (v36);
\draw (v46) -- (v37);
\node at (1.5,0) {$a$};
\node at (1,3.1) {$b$};
\node at (2,3) {$c$};
\node at (5.5,0) {$a$};
\node at (7,3.1) {$b$};
\node at (8,3) {$a$};
\node at (11.5,0) {$a$};
\node at (12,3.1) {$b$};
\node at (15,3) {$c$};
\node at (18,0) {$a$};
\node at (17.5,3.1) {$b$};
\node at (18.5,3) {$c$};
\node at (-2.5,-1) {$\mathcal{S}_1=\mathcal{Q}_1$};
\node at (1,-1) {$\mathcal{Q}_2$};
\node at (5,-1) {$\mathcal{Q}_3$};
\node at (11,-1) {$\mathcal{Q}_4$};
\node at (18.5,-1) {$\mathcal{Q}_5=\mathcal{P}$};
\end{tikzpicture}$$
\caption{Construction sequence of $\mathcal{P}$}\label{fig:toral}
\end{figure}
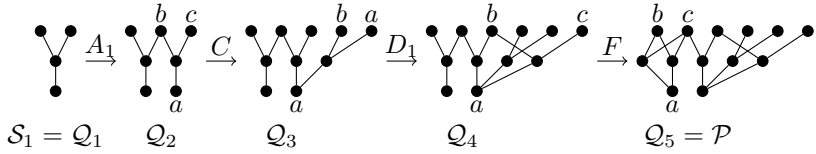
\end{Ex}

As a direct consequence of Remark 2 in \textbf{\cite{CM}} we get the following result.

\begin{theorem}\label{thm:indform}
If $\mathcal{P}$ is a toral poset, then $\ind\mathfrak{g}_A(\mathcal{P})=|Rel_E(\mathcal{P})|-|Ext(\mathcal{P})|+1$.
\end{theorem}

\begin{theorem}\label{thm:wedge}
If $\mathcal{P}$ is a toral poset, then $\Sigma(\mathcal{P})$ is homotopic to a wedge sum of $\ind\mathfrak{g}_A(\mathcal{P})$ one-spheres.
\end{theorem}
\begin{proof}
Let $\mathcal{P}$ be a toral poset constructed from the toral-pairs $\{(\mathcal{S}_i,F_{\mathcal{S}_{i}})\}_{i=1}^n$ with construction sequence $\mathcal{S}_1=\mathcal{Q}_1\subset\mathcal{Q}_2\subset\hdots\subset\mathcal{Q}_n=\mathcal{P}$. Since each $\Sigma(\mathcal{S}_i)$ is contractible, for $i=1,\hdots,n$, $\Sigma(\mathcal{S}_i)$ is homotopic to the Hasse diagram of $\mathcal{S}_{i_{Rel_E(\mathcal{S}_i)}}$. Performing this homotopy sequentially for each $\Sigma(\mathcal{S}_i)\subset\Sigma(\mathcal{P})$ from $i=1$ to $n$, we arrive at the simplicial complex $\Sigma(\mathcal{P})'$. Note that $\Sigma(\mathcal{P})'$ is just the Hasse diagram of $\mathcal{P}_{Rel_E(\mathcal{P})}$. As $\Sigma(\mathcal{P})'$ is a connected graph with $|Rel_E(\mathcal{P})|$ edges and $|Ext(\mathcal{P})|$ vertices, $\Sigma(\mathcal{P})'$ is homotopic to a wedge sum of $|Rel_E(\mathcal{P})|-|\mathcal{P}|+1$ one-spheres. Thus, considering Theorem~\ref{thm:indform}, the result follows.
\end{proof}

\begin{theorem}
Let $\mathcal{P}$ be a toral poset constructed from the toral-pairs $\{(\mathcal{S}_i,F_{\mathcal{S}_{i}})\}_{i=1}^n$ with construction sequence $\mathcal{S}_1=\mathcal{Q}_1\subset\mathcal{Q}_2\subset\hdots\subset\mathcal{Q}_n=\mathcal{P}$. Then $\mathcal{P}$ is Frobenius if and only if $\mathcal{Q}_i$ is formed from $\mathcal{Q}_{i-1}$ and $\mathcal{S}_i$ by applying rules from the set $\{\textup{A}_1, \textup{A}_2, \textup{C}, \textup{D}_1, \textup{D}_2, \textup{F}\}$, for $i=2,\hdots,n$.
\end{theorem}
\begin{proof}
Using Theorem~\ref{thm:indform}, for $i=2,\hdots,n$, if $\mathcal{Q}_{i}$ is formed from $\mathcal{Q}_{i-1}$ and $\mathcal{S}_i$ by applying rule 
\begin{itemize}
    \item $\textup{A}_1$, $\textup{A}_2$, or $\textup{C}$, then $\ind\mathfrak{g}_A(\mathcal{Q}_{i})=\ind\mathfrak{g}_A(\mathcal{Q}_{i-1})+2-2=\ind\mathfrak{g}_A(\mathcal{Q}_{i-1});$
    \item $\textup{B}$, $\textup{E}_1$, or $\textup{E}_2$, then $\ind\mathfrak{g}_A(\mathcal{Q}_{i})=\ind\mathfrak{g}_A(\mathcal{Q}_{i-1})+2-1=\ind\mathfrak{g}_A(\mathcal{Q}_{i-1})+1;$
    \item $\textup{D}_1$ or $\textup{D}_2$, then $\ind\mathfrak{g}_A(\mathcal{Q}_{i})=\ind\mathfrak{g}_A(\mathcal{Q}_{i-1})+1-1=\ind\mathfrak{g}_A(\mathcal{Q}_{i-1});$
    \item $\textup{F}$, then $\ind\mathfrak{g}_A(\mathcal{Q}_{i})=\ind\mathfrak{g}_A(\mathcal{Q}_{i-1});$
    \item $\textup{G}_1$ or $\textup{G}_2$, then $\ind\mathfrak{g}_A(\mathcal{Q}_{i})=\ind\mathfrak{g}_A(\mathcal{Q}_{i-1})+1;$
    \item $\textup{H}$, then $\ind\mathfrak{g}_A(\mathcal{Q}_{i})=\ind\mathfrak{g}_A(\mathcal{Q}_{i-1})+2.$
\end{itemize}
Thus, as $\ind\mathfrak{g}_A(\mathcal{Q}_1)=\ind\mathfrak{g}_A(\mathcal{S}_1)=0$, the result follows.
\end{proof}

\section{Toral functionals and spectrum}\label{sec:spec}

In this section, given a Frobenius, toral poset $\mathcal{P}$ constructed from the toral-pairs $\{(\mathcal{S}_i,F_{\mathcal{S}_{i}})\}_{i=1}^n$, we provide an inductive procedure for constructing a Frobenius functional $F_{\mathcal{P}}\in(\mathfrak{g}_A(\mathcal{P}))^*$ from the functionals $F_{\mathcal{S}_{i}}$, for $i=1,\hdots,n$. Coincidentally, we obtain an alternative proof that toral posets formed by applying rules from the set $\{\textup{A}_1, \textup{A}_2, \textup{C}, \textup{D}_1, \textup{D}_2,\textup{F}\}$ are Frobenius (see Theorem~\ref{thm:toralFun}). Furthermore, we characterize the spectrum of all Frobenius, type-A Lie poset algebras which correspond to toral posets (see Theorem~\ref{thm:toralspec}). 

\begin{remark}
Let $\mathcal{P}$ be a toral poset constructed from the toral-pairs $\{(\mathcal{S}_i,F_{\mathcal{S}_{i}})\}_{i=1}^n$ with construction sequence $\mathcal{S}_1=\mathcal{Q}_1\subset\mathcal{Q}_2\subset\hdots\subset\mathcal{Q}_n=\mathcal{P}$. Throughout this section, in the notation of Section~\ref{sec:toralpos}, if $\mathcal{Q}=\mathcal{Q}_{i-1}$ and $\mathcal{S}=\mathcal{S}_i$, for $i=2,\hdots,n$, then we denote the elements of $\{a,x\}$ by $x_i$, $\{b,y\}$ by $y_i$, and $\{c,z\}$ by $z_i$.
\end{remark}

\begin{definition}\label{def:toralfun}
If $\mathcal{P}$ is a Frobenius, toral poset constructed from the toral-pairs $\{(\mathcal{S}_i,F_{\mathcal{S}_{i}})\}_{i=1}^n$ with construction sequence  $\mathcal{S}_1=\mathcal{Q}_1\subset\mathcal{Q}_2\subset\hdots\subset\mathcal{Q}_n=\mathcal{P}$, then define the ``toral" functional $F_{\mathcal{Q}_i}\in(\mathfrak{g}_A(\mathcal{Q}_i))^*$, for $i=1,\hdots,n$, as follows:
\begin{itemize}[label={\large\textbullet}]
    \item $F_{\mathcal{Q}_{1}}=F_{\mathcal{S}_{1}}$;
    \item if $\mathcal{Q}_{i}$ is formed from $\mathcal{Q}_{i-1}$ and $\mathcal{S}_i$, for $1<i\le n$, by applying rule
    \begin{itemize}[label={\small\textbullet}]
        \item $\textup{A}_1,\textup{A}_2$, or $\textup{C}$, then $$F_{\mathcal{Q}_i}=F_{\mathcal{Q}_{i-1}}+F_{\mathcal{S}_{i}}.$$
        \item $\textup{D}_1$, then
        \[F_{\mathcal{Q}_i} =  \begin{cases} 
      F_{\mathcal{Q}_{i-1}}+F_{\mathcal{S}_{i}}-E^*_{x_i,y_i}, & \mathcal{S}_i\text{ has one minimal element}; \\
                                                &                        \\
       F_{\mathcal{Q}_{i-1}}+F_{\mathcal{S}_{i}}-E^*_{y_i,x_i}, & \mathcal{S}_i\text{ has one maximal element}.
   \end{cases}
\]
        \item $\textup{D}_2$, then
        \[F_{\mathcal{Q}_i} =  \begin{cases} 
      F_{\mathcal{Q}_{i-1}}+F_{\mathcal{S}_{i}}-E^*_{x_i,z_i}, & \mathcal{S}_i\text{ has one minimal element}; \\
                                                &                        \\
       F_{\mathcal{Q}_{i-1}}+F_{\mathcal{S}_{i}}-E^*_{z_i,x_i}, & \mathcal{S}_i\text{ has one maximal element}.
   \end{cases}
\]
        \item $\textup{F}$, then
        \[F_{\mathcal{Q}_i} =  \begin{cases} 
      F_{\mathcal{Q}_{i-1}}+F_{\mathcal{S}_{i}}-E^*_{x_i,y_i}-E^*_{x_i,z_i}, & \mathcal{S}_i\text{ has one minimal element}; \\
                                                &                        \\
       F_{\mathcal{Q}_{i-1}}+F_{\mathcal{S}_{i}}-E^*_{y_i,x_i}-E^*_{z_i,x_i}, & \mathcal{S}_i\text{ has one maximal element}.
   \end{cases}
\]
    \end{itemize}
\end{itemize}
\end{definition}

\begin{remark}\label{rem:finf}
Note that $E^*_{p,q}$ is a summand of $F_{\mathcal{Q}_i}$ if and only if $E^*_{p,q}$ is a summand of $F_{\mathcal{Q}_{i-1}}$ or $F_{\mathcal{S}_i}$.
\end{remark}

The following Lemma is an immediate consequence of Definition~\ref{def:toralfun}.

\begin{lemma}\label{lem:toralUD}
If $\mathcal{P}$ is a toral poset constructed from the toral-pairs $\{(\mathcal{S}_i,F_{\mathcal{S}_i})\}_{i=1}^n$ with the construction sequence $\mathcal{S}_1=\mathcal{Q}_1\subset\mathcal{Q}_2\subset\hdots\subset\mathcal{Q}_n=\mathcal{P}$, then
\begin{itemize}
    \item $F_{\mathcal{Q}_i}$ is small,
    \item $D_{F_{\mathcal{Q}_i}}(\mathcal{Q}_i)$ is an order ideal of $\mathcal{Q}_i$, $U_{F_{\mathcal{Q}_i}}(\mathcal{Q}_i)$ is a filter, and $B_{F_{\mathcal{Q}_i}}(\mathcal{Q}_i)=\emptyset$,
\end{itemize} 
for $i=1,\hdots,n$.
\end{lemma}

\begin{remark}
Recall that for a poset $\mathcal{P}$, elements of $\mathfrak{g}_A(\mathcal{P})$ are of the form $$\sum_{(p_i, p_j)\in Rel(\mathcal{P})}c_{p_i,p_j}E_{p_i,p_j}+\sum_{p_i\in\mathcal{P}}c_{p_i,p_i}E_{p_i,p_i},$$ where $\sum_{p_i\in\mathcal{P}}c_{p_i,p_i}=0$. Let $\mathcal{P}$ be a poset formed by combining the posets $\mathcal{S}$ and $\mathcal{Q}$ by identifying minimal elements or maximal elements. If $B\in\mathfrak{g}_A(\mathcal{P})$, then let $B|_{\mathcal{Q}}$ denote the restriction of $B$ to basis elements of $\mathfrak{g}_A(\mathcal{Q})$ and $B|_{\mathcal{S}}$ denote the restriction of $B$ to basis elements of $\mathfrak{g}_A(\mathcal{S})$.
\end{remark}

\begin{lemma}\label{lem:bl1}
If $\mathcal{P}$ is a Frobenius, toral poset and $B\in\mathfrak{g}(\mathcal{P})$ satisfies $F_{\mathcal{P}}([E_{p,p},B])=0$, for all $p\in\mathcal{P}$, then $E^*_{p,q}(B)=0$, for $E^*_{p,q}$ a summand of $F_{\mathcal{P}}$.
\end{lemma}
\begin{proof}
Assume $\mathcal{P}$ is constructed from the toral-pairs $\{(\mathcal{S}_i,F_{\mathcal{S}_{i}})\}_{i=1}^n$ with construction sequence $\mathcal{S}_1=\mathcal{Q}_1\subset\mathcal{Q}_2\subset\hdots\subset\mathcal{Q}_n=\mathcal{P}$. The proof is by induction on $i$. Throughout, we assume that $\mathcal{S}_i$, for $i=1,\hdots,n$, contains a single minimal element; the dual case follows via a similar argument. By property \textbf{(F1)} of toral-pairs, we know that $\Gamma_{F_{\mathcal{S}_{i}}}$ is a tree, for $i=1,\hdots,n$. Thus, for $i=1,\hdots,n$, we can apply the following inductive procedure, denoted $Proc(i)$:
\\*

\noindent
\textbf{Step 1:} Consider all degree-one vertices $p_1\in\mathcal{S}_i\backslash Ext(\mathcal{S}_i)$ of $\Gamma_{F_{\mathcal{S}_{i}}}=\Gamma_1$. Suppose $p_1$ is adjacent to $q_1$ in $\Gamma_1$. Then $$F_{\mathcal{Q}_{i}}([E_{p_1,p_1},B])=F_{\mathcal{S}_{i}}([E_{p_1,p_1},B|_{\mathcal{S}_i}])=E^*_{p_1,q_1}(B)=0\text{ }(\text{or }-E^*_{q_1,p_1}(B)=0),$$ where $E^*_{p_1,q_1}$ (or $E^*_{q_1,p_1}$) is a summand of $F_{\mathcal{P}}$. Remove such $p_1$ and edges $(p_1,q_1)$ from $\Gamma_1$ to form the directed graph $\Gamma_2$.
\\*

\noindent
\textbf{Step j:} Consider all degree-one vertices $p_j\in\mathcal{S}_i\backslash Ext(\mathcal{S}_i)$ of $\Gamma_{j}$. Suppose $p_j$ is adjacent to $q_j$ in $\Gamma_{j}$. Using the results of $\mathbf{Step\text{ }1}$ through $\mathbf{Step\text{ }j-1}$, $$F_{\mathcal{Q}_{i}}([E_{p_j,p_j},B])=F_{\mathcal{S}_{i}}([E_{p_j,p_j},B|_{\mathcal{S}_i}])=E^*_{p_j,q_j}(B)=0\text{ }(\text{or }-E^*_{q_j,p_j}(B)=0),$$ where $E^*_{p_j,q_j}$ (or $E^*_{q_j,p_j}$) is a summand of $F_{\mathcal{P}}$. Remove such $p_j$ and edges $(p_j,q_j)$ from $\Gamma_{j-1}$ to form the directed graph $\Gamma_j$.
\\*

\noindent
By properties \textbf{(F1)} and \textbf{(F3)} of toral-pairs and the fact that $\mathcal{S}_i$ is finite, there must exist a finite $m_i$ for which $\Gamma_{m_i}$ consists solely of the elements of $Ext(\mathcal{S}_i)$ along with the edges between them. Note that this implies $E^*_{p,q}(B)=0$, for $E^*_{p,q}$ a summand of $F_{\mathcal{Q}_i}$ with $p,q\in\mathcal{S}_i$ and either $p\in\mathcal{S}_i\backslash Ext(\mathcal{S}_i)$ or $q\in\mathcal{S}_i\backslash Ext(\mathcal{S}_i)$.

For the base case, $i=1$, $\mathcal{Q}_1=\mathcal{S}_1$ has maximal elements $y_1,z_1$ and minimal element $x_1$. Applying $Proc(1)$ it remains to consider $E^*_{x_1,y_1}(B)$ and $E^*_{x_1,z_1}(B)$; but the implications of $Proc(1)$ allow us to conclude that $$F_{\mathcal{Q}_1}([E_{y_1,y_1},B])=-E^*_{x_1,y_1}(B)=0$$ and $$F_{\mathcal{Q}_1}([E_{z_1,z_1},B])=-E^*_{x_1,z_1}(B)=0.$$ The base of the induction is thus established. 

Now, assume the result holds for $B\in\mathfrak{g}(\mathcal{Q}_{i-1})$, for $1<i\le n$. There are four cases to consider, based on the rules used in the construction sequence of $\mathcal{P}$.
\\*

\noindent
\textbf{Case 1:} $\mathcal{Q}_i$ is formed from $\mathcal{Q}_{i-1}$ and $\mathcal{S}_i$ by applying rule $\textup{A}_1$ or $\textup{A}_2$. Without loss of generality, assume $\mathcal{Q}_i$ is formed by applying rule $\textup{A}_1$. The implications of $Proc(i)$ allow us to conclude that $$F_{\mathcal{Q}_i}([E_{z_i,z_i},B])=F_{\mathcal{S}_i}([E_{z_i,z_i},B])=-E^*_{x_i,z_i}(B)=0,$$ so that $$F_{\mathcal{Q}_i}([E_{x_i,x_i},B])=F_{\mathcal{S}_i}([E_{x_i,x_i},B])=E^*_{x_i,y_i}(B)=0.$$ Thus, $E^*_{p,q}(B)=0$, for $E^*_{p,q}$ a summand of $F_{\mathcal{Q}_i}$ with $p,q\in\mathcal{S}_i$ and either $p\in\mathcal{S}_i\backslash \mathcal{Q}_i$ or $q\in\mathcal{S}_i\backslash \mathcal{Q}_i$. Considering Definition~\ref{def:toralfun}, $F_{\mathcal{Q}_i}([E_{p,p},B])=0$ reduces to $F_{\mathcal{Q}_{i-1}}([E_{p,p},B|_{\mathcal{Q}_{i-1}}])=0$, for $p\in\mathcal{Q}_{i-1}$. By the inductive hypothesis this implies that $E^*_{p,q}(B)=0$, for $E^*_{p,q}$ a summand of $F_{\mathcal{Q}_i}$.
\\*

\noindent
\textbf{Case 2:} $\mathcal{Q}_i$ is formed from $\mathcal{Q}_{i-1}$ and $\mathcal{S}_i$ by applying rule $\textup{C}$. The implications of $Proc(i)$ allow us to conclude that
$$F_{\mathcal{Q}_i}([E_{y_i,y_i},B])=F_{\mathcal{S}_i}([E_{y_i,y_i},B])=-E^*_{x_i,y_i}(B)=0$$ and $$F_{\mathcal{Q}_i}([E_{z_i,z_i},B])=F_{\mathcal{S}_i}([E_{z_i,z_i},B])=-E^*_{x_i,z_i}(B)=0.$$ Thus, $E^*_{p,q}(B)=0$, for $E^*_{p,q}$ a summand of $F_{\mathcal{Q}_i}$ with $p,q\in\mathcal{S}_i$ and either $p\in\mathcal{S}_i\backslash \mathcal{Q}_i$ or $q\in\mathcal{S}_i\backslash \mathcal{Q}_i$. Considering Definition~\ref{def:toralfun}, $F_{\mathcal{Q}_i}([E_{p,p},B])=0$ reduces to $F_{\mathcal{Q}_{i-1}}([E_{p,p},B|_{\mathcal{Q}_{i-1}}])=0$, for $p\in\mathcal{Q}_{i-1}$. By the inductive hypothesis this implies that $E^*_{p,q}(B)=0$, for $E^*_{p,q}$ a summand of $F_{\mathcal{Q}_i}$.
\\*

\noindent
\textbf{Case 3:} $\mathcal{Q}_i$ is formed from $\mathcal{Q}_{i-1}$ and $\mathcal{S}_i$ by applying rule $\textup{D}_1$ or $\textup{D}_2$. Without loss of generality, assume $\mathcal{Q}_i$ is formed by applying rule $\textup{D}_1$. The implications of $Proc(i)$ allow us to conclude that $$F_{\mathcal{Q}_i}([E_{z_i,z_i},B])=F_{\mathcal{S}_i}([E_{z_i,z_i},B])=-E^*_{x_i,z_i}(B)=0.$$ Thus, $E^*_{p,q}(B)=0$, for $E^*_{p,q}$ a summand of $F_{\mathcal{Q}_i}$ with $p,q\in\mathcal{S}_i$ and either $p\in\mathcal{S}_i\backslash \mathcal{Q}_i$ or $q\in\mathcal{S}_i\backslash \mathcal{Q}_i$. Considering Definition~\ref{def:toralfun}, $F_{\mathcal{Q}_i}([E_{p,p},B])=0$ reduces to $F_{\mathcal{Q}_{i-1}}([E_{p,p},B|_{\mathcal{Q}_{i-1}}])=0$, for $p\in\mathcal{Q}_{i-1}$. By the inductive hypothesis this implies that $E^*_{p,q}(B)=0$, for $E^*_{p,q}$ a summand of $F_{\mathcal{Q}_i}$.
\\*

\noindent
\textbf{Case 4:} $\mathcal{Q}_i$ is formed from $\mathcal{Q}_{i-1}$ and $\mathcal{S}_i$ by applying rule $\textup{F}$. The implications of $Proc(i)$ allow us to conclude that $E^*_{p,q}(B)=0$, for $E^*_{p,q}$ a summand of $F_{\mathcal{Q}_i}$ with $p,q\in\mathcal{S}_i$ and either $p\in\mathcal{S}_i\backslash \mathcal{Q}_i$ or $q\in\mathcal{S}_i\backslash \mathcal{Q}_i$. Considering Definition~\ref{def:toralfun}, $F_{\mathcal{Q}_i}([E_{p,p},B])=0$ reduces to $F_{\mathcal{Q}_{i-1}}([E_{p,p},B|_{\mathcal{Q}_{i-1}}])=0$, for $p\in\mathcal{Q}_{i-1}$. By the inductive hypothesis this implies that $E^*_{p,q}(B)=0$, for $E^*_{p,q}$ a summand of $F_{\mathcal{Q}_i}$.
\\*

\noindent
The induction establishes the result.
\end{proof}

\begin{lemma}\label{lem:bl2}
Let $\mathcal{P}$ be a Frobenius, toral poset constructed from the toral-pairs $\{(\mathcal{S}_i,F_i)\}_{i=1}^n$ with construction sequence  $\mathcal{S}_1=\mathcal{Q}_1\subset\mathcal{Q}_2\subset\hdots\subset\mathcal{Q}_n=\mathcal{P}$. If $1<i\le n$, then $B\in\mathfrak{g}(\mathcal{Q}_i)\cap\ker(F_{\mathcal{Q}_i})$ must satisfy $B|_{\mathcal{Q}_{i-1}}\in\mathfrak{g}(\mathcal{Q}_{i-1})\cap\ker(F_{\mathcal{Q}_{i-1}})$ and $B|_{\mathcal{S}_i}\in\mathfrak{g}(\mathcal{S}_i)\cap\ker(F_{\mathcal{S}_i})$.
\end{lemma}
\begin{proof}
We assume throughout that $\mathcal{S}_i$, for $i=2,\hdots,n$, has a single minimal element, the other case following via a symmetric argument. Take $B\in\mathfrak{g}(\mathcal{Q}_{i})\cap\ker(F_{\mathcal{Q}_{i}})$, for $1<i\le n$. Considering Remark~\ref{rem:finf}, the equation $F_{\mathcal{Q}_{i-1}}([E_{p,p},B|_{\mathcal{Q}_{i-1}}])=0$ (resp. $F_{\mathcal{S}_{i}}([E_{p,p},B|_{\mathcal{S}_{i}}])=0$) consists of terms of the form $E^*_{p,q}(B)$ and $E^*_{r,p}(B)$, for summands $E^*_{p,q}$ and $E^*_{r,p}$ of $F_{\mathcal{Q}_i}$ such that $p,q,r\in\mathcal{Q}_{i-1}$ (resp. $p,q,r\in\mathcal{S}_i$). Thus, by Lemma~\ref{lem:bl1}, $B\in\mathfrak{g}(\mathcal{Q}_i)\cap\ker(F_{\mathcal{Q}_i})$ must satisfy 
$$F_{\mathcal{Q}_{i-1}}([E_{p,p},B|_{\mathcal{Q}_{i-1}}])=0\text{ for }p\in\mathcal{Q}_{i-1}$$ and $$F_{\mathcal{S}_{i}}([E_{p,p},B|_{\mathcal{S}_{i}}])=0\text{ for }p\in\mathcal{S}_{i}.$$
It remains to consider restrictions placed on $B$ by equations of the form $F_{\mathcal{Q}_{i}}([E_{p,q},B])=0$ for $p,q\in\mathcal{Q}_{i-1}$ (resp. $p,q\in\mathcal{S}_{i}$); but combining Remark~\ref{rem:finf} with the fact that $\mathcal{Q}_{i-1}\cap\mathcal{S}_i\subset Ext(\mathcal{Q}_i)$, it is immediate that 
$$F_{\mathcal{Q}_{i-1}}([E_{p,q},B|_{\mathcal{Q}_{i-1}}])=F_{\mathcal{Q}_{i}}([E_{p,q},B])=0\text{ for }p,q\in\mathcal{Q}_{i-1}$$ and $$F_{\mathcal{S}_{i}}([E_{p,q},B|_{\mathcal{S}_{i}}])=F_{\mathcal{Q}_{i}}([E_{p,q},B])=0\text{ for }p,q\in\mathcal{S}_{i}.$$ Thus, the result follows.
\end{proof}

\begin{theorem}\label{thm:toralFun}
If $\mathcal{P}$ is a Frobenius, toral poset constructed from the toral-pairs $\{(\mathcal{S}_i,F_{\mathcal{S}_i})\}_{i=1}^n$ with construction sequence  $\mathcal{S}_1=\mathcal{Q}_1\subset\mathcal{Q}_2\subset\hdots\subset\mathcal{Q}_n=\mathcal{P}$, then $F_{\mathcal{Q}_i}\in(\mathfrak{g}_A(\mathcal{Q}_i))^*$ is Frobenius, for $i=1,\hdots,n$.
\end{theorem}
\begin{proof}
We will show by induction on $i$, that $B\in\mathfrak{g}(\mathcal{Q}_i)\cap\ker(F_{\mathcal{Q}_i})$ satisfies
\begin{itemize}
    \item $E^*_{p,p}(B)=E^*_{q,q}(B)$ for all $p,q\in\mathcal{Q}_i$, and
    \item $E^*_{p,q}(B)=0$ for all $p,q\in\mathcal{Q}_i$ satisfying $p\preceq q$,
\end{itemize}
for $i=1,\hdots,n$. Since $B\in\mathfrak{sl}(|\mathcal{Q}_i|)$, the result will follow. 

For $i=1$, the result is clear since $(\mathcal{Q}_1,F_{\mathcal{Q}_1})=(\mathcal{S}_1,F_{\mathcal{S}_1})$ is a toral-pair. Assume the result holds for $B\in\mathfrak{g}(\mathcal{Q}_{i-1})\cap\ker(F_{\mathcal{Q}_{i-1}})$, for $1<i\le n$. Take $B\in\mathfrak{g}(\mathcal{Q}_i)\cap\ker(F_{\mathcal{Q}_i})$. By Lemma~\ref{lem:bl2}, we know that 
\begin{enumerate}[label=\textup(\alph*\textup)]
        \item $B|_{\mathcal{Q}_{i-1}}\in\mathfrak{g}_A(\mathcal{Q}_{i-1})\cap\ker(F_{\mathcal{Q}_{i-1}})$, and
        \item $B|_{\mathcal{S}_{i}}\in\mathfrak{g}_A(\mathcal{S}_{i})\cap\ker(F_{\mathcal{S}_{i}})$.
\end{enumerate}
Combining (a) with the inductive hypothesis, we find that $E^*_{p,p}(B)=E^*_{q,q}(B)$, for all $p,q\in\mathcal{Q}_{i-1}$, and $E^*_{p,q}(B)=0$, for all $p,q\in\mathcal{Q}_{i-1}$ satisfying $p\preceq_{\mathcal{Q}_{i-1}} q$. Furthermore, combining (b) with property \textbf{(F4)} of toral-pairs, we must also have that $E^*_{p,p}(B)=E^*_{q,q}(B)$, for all $p,q\in\mathcal{S}_{i}$, and $E^*_{p,q}(B)=0$, for all $p,q\in\mathcal{S}_i$ satisfying $p\preceq_{\mathcal{S}_i} q$. Thus, $E^*_{p,q}(B)=0$, for all $p,q\in\mathcal{Q}_i$ satisfying $p\preceq_{\mathcal{Q}_i} q$, and, since $\mathcal{Q}_i$ is connected, $E^*_{p,p}(B)=E^*_{q,q}(B)$, for all $p,q\in\mathcal{Q}_{i}$. Hence, since $B\in\mathfrak{sl}(|\mathcal{Q}_i|)$, we find that $B=0$ and the result follows.
\end{proof}

\begin{theorem}\label{thm:toralspec}
If $\mathcal{P}$ is a Frobenius, toral poset, then the spectrum of $\mathfrak{g}_A(\mathcal{P})$ is binary.
\end{theorem}
\begin{proof}
Assume $\mathcal{P}$ is constructed from the toral-pairs $\{(\mathcal{S}_i,F_{\mathcal{S}_i})\}_{i=1}^n$ with construction sequence $S_1=\mathcal{Q}_1\subset\mathcal{Q}_2\subset\hdots\subset\mathcal{Q}_n=\mathcal{P}$. We will show that $\mathfrak{g}_A(\mathcal{Q}_i)$, for $1\le i\le n$, has binary spectrum by induction on $i$. Throughout, we assume $\mathcal{S}_i$ has a single minimal element, the other case following via a symmetric argument.  

For the base case, $\mathfrak{g}_A(\mathcal{Q}_1)=\mathfrak{g}_A(\mathcal{S}_1)$ has a binary spectrum by property \textbf{(P2)} of toral-pairs. 
So, assume that $\mathfrak{g}_A(\mathcal{Q}_{i-1})$, for $1<i\le n$, has a binary spectrum. By Lemma~\ref{lem:toralUD}, the form of $\widehat{F}_{\mathcal{Q}_i}$ is given by Theorem~\ref{thm:spec01}. Thus, considering the proof of Theorem~\ref{thm:01}, to determine the spectrum of $\mathfrak{g}_A(\mathcal{P})$ it suffices to calculate the values $[\widehat{F},x]$, for $x\in \mathscr{B}_{\mathcal{Q}_{i},F_{\mathcal{Q}_{i}}}$. Recall, that $E_{p,p}-E_{q,q}\in \mathscr{B}_{\mathcal{Q}_{i},F_{\mathcal{Q}_{i}}}$ is an eigenvector of $[\widehat{F},-]$ with eigenvalue 0, $E_{p,q}\in \mathscr{B}_{\mathcal{Q}_{i},F_{\mathcal{Q}_{i}}}$ with $p,q\in D_{F_{\mathcal{Q}_i}}(\mathcal{Q}_i)$ or $p,q\in U_{F_{\mathcal{Q}_i}}(\mathcal{Q}_i)$ has eigenvalue 0, and $E_{p,q}\in \mathscr{B}_{\mathcal{Q}_{i},F_{\mathcal{Q}_{i}}}$ with $p\in D_{F_{\mathcal{Q}_i}}(\mathcal{Q}_i)$ and $q\in U_{F_{\mathcal{Q}_i}}(\mathcal{Q}_i)$ has eigenvalue 1. Considering Definition~\ref{def:toralfun}, we must have that $D_{F_{\mathcal{Q}_{i-1}}}(\mathcal{Q}_{i-1})\subset D_{F_{\mathcal{Q}_i}}(\mathcal{Q}_i)$ and $U_{F_{\mathcal{Q}_{i-1}}}(\mathcal{Q}_{i-1})\subset U_{F_{\mathcal{Q}_i}}(\mathcal{Q}_i)$ so that, by the inductive hypothesis, the elements of $\mathscr{B}_{\mathcal{Q}_{i-1},F_{\mathcal{Q}_{i-1}}}\subset\mathscr{B}_{\mathcal{Q}_i,F_{\mathcal{Q}_i}}$ contribute an equal number of 0's and 1's to the spectrum of $\mathfrak{g}_A(\mathcal{Q}_i)$. It remains to consider the eigenvalues corresponding to the elements of $\mathscr{B}_{\mathcal{Q}_i,F_{\mathcal{Q}_i}}\backslash \mathscr{B}_{\mathcal{Q}_{i-1},F_{\mathcal{Q}_{i-1}}}$. The collection of such basis elements breaks into four cases:
\\*

\noindent
\textbf{Case 1:} $\mathcal{Q}_{i}$ is formed from $\mathcal{Q}_{i-1}$ and $\mathcal{S}_i$ by applying rule $\textup{A}_1$, $\textup{A}_2$, or $\textup{C}$. In this case, 
$$\mathscr{B}_{\mathcal{Q}_i,F_{\mathcal{Q}_i}}\backslash \mathscr{B}_{\mathcal{Q}_{i-1},F_{\mathcal{Q}_{i-1}}}=\mathscr{B}_{\mathcal{S}_i,F_{\mathcal{S}_i}}.$$
By Definition~\ref{def:toralfun}, we have  $D_{F_{S_i}}(\mathcal{S}_{i})\subset D_{F_{\mathcal{Q}_i}}(\mathcal{Q}_i)$ and $U_{F_{S_i}}(\mathcal{S}_{i})\subset U_{F_{\mathcal{Q}_i}}(\mathcal{Q}_i)$. Thus, by property \textbf{(P2)} of toral-pairs, we must have that the elements of $\mathscr{B}_{\mathcal{S}_i,F_{\mathcal{S}_i}}$ contribute an equal number of 0's and 1's to the spectrum of $\mathfrak{g}_A(\mathcal{Q}_i)$. Thus, $\mathfrak{g}_A(\mathcal{Q}_i)$ has a binary spectrum.
\\*

\noindent
\textbf{Case 2:} $\mathcal{Q}_{i}$ is formed from $\mathcal{Q}_{i-1}$ and $\mathcal{S}_i$ by applying rule $\textup{D}_1$ or $\textup{D}_2$. Assume, without loss of generality, that $\mathcal{Q}_i$ is formed from $\mathcal{Q}_{i-1}$ and $\mathcal{S}_i$ by applying rule $\textup{D}_1$. In this case, 
$$\mathscr{B}_{\mathcal{Q}_i,F_{\mathcal{Q}_i}}\backslash \mathscr{B}_{\mathcal{Q}_{i-1},F_{\mathcal{Q}_{i-1}}}=\mathscr{B}_{\mathcal{S}_i,F_{\mathcal{S}_i}}\backslash\{E_{x,x}-E_{y,y},E_{x,y}\}.$$ As in Case 1, the basis elements of $\mathscr{B}_{\mathcal{S}_i,F_{\mathcal{S}_i}}$ contribute an equal number of 0's and 1's to the spectrum of $\mathfrak{g}_A(\mathcal{Q}_i)$. Since $E_{x,x}-E_{y,y}$ is an eigenvector of $ad(\widehat{F}_{\mathcal{Q}_i})$ with eigenvalue 0, while $E_{x,y}$ is an eigenvector with eigenvalue 1, it follows that $\mathfrak{g}_A(\mathcal{Q}_i)$ has a binary spectrum.
\\*

\noindent
\textbf{Case 3:} $\mathcal{Q}_{i}$ is formed from $\mathcal{Q}_{i-1}$ and $\mathcal{S}_i$ by applying rule $\textup{F}$. In this case, 
$$\mathscr{B}_{\mathcal{Q}_i,F_{\mathcal{Q}_i}}\backslash \mathscr{B}_{\mathcal{Q}_{i-1},F_{\mathcal{Q}_{i-1}}}=\mathscr{B}_{\mathcal{S}_i,F_{\mathcal{S}_i}}\backslash\{E_{x,x}-E_{y,y},E_{x,x}-E_{z,z},E_{x,y},E_{x,z}\}.$$
As in Case 1, the basis elements of $\mathscr{B}_{\mathcal{S}_i,F_{\mathcal{S}_i}}$ contribute an equal number of 0's and 1's to the spectrum of $\mathfrak{g}_A(\mathcal{Q}_i)$. Since $E_{x,x}-E_{y,y}$ and $E_{x,x}-E_{z,z}$ are eigenvectors of $ad(\widehat{F}_{\mathcal{Q}_i})$ with eigenvalue 0, while $E_{x,y}$ and $E_{x,z}$ are eigenvectors with eigenvalue 1, it follows that $\mathfrak{g}_A(\mathcal{Q}_i)$ has a binary spectrum.
\end{proof}

As a corollary to Theorem~\ref{thm:toralspec}, in conjunction with Theorems 10 and 12 of \textbf{\cite{CM}} yields the following succinct result.

\begin{theorem} \label{thm:end}
If $\mathcal{P}$ is a Frobenius poset of height at most two, then $\mathfrak{g}_A(\mathcal{P})$ is toral -- and so has a binary spectrum.
\end{theorem}

\begin{remark}
Extensive calculations suggest that Theorem~\ref{thm:end} is true for posets of arbitrary height.
\end{remark}

\section{Appendix A}

In this appendix, we show that given a finite poset $\mathcal{P}$ and $F\in(\mathfrak{g}_A(\mathcal{P}))^*$, that the system of equations $F([E_{p_1,p_1}-E_{p_i,p_i},B])=0$ for $p_1\in\mathcal{P}$ and all $p_1\neq p_i\in\mathcal{P}$ is equivalent to the system $F([E_{p_i,p_i},B])=0$ for all $p_i\in\mathcal{P}$. 

The conditions $$F([E_{p_1,p_1}-E_{p_i,p_i},B])=F([E_{p_1,p_1},B])-F([E_{p_i,p_i},B])=0,$$ for all $p_i\in\mathcal{P}$ such that $p_1\neq p_i$, imply that $$F([E_{p_i,p_i},B])=C$$ for all $p_i\in\mathcal{P}$ and $C\in\textbf{k}$. Thus, $$\sum_{p_i\in\mathcal{P}}F([E_{p_i,p_i},B])=nC;$$ but $$\sum_{p_i\in\mathcal{P}}F([E_{p_i,p_i},B])=0.$$ To see this, note that if $E_{p_j,p_k}^*(B)$ is a term of $\sum_{p_i\in\mathcal{P}}F([E_{p_i,p_i},B])$ (corresponding to $F([E_{p_j,p_j},B])$), then $-E_{p_j,p_k}^*(B)$ must also be a term (coming from $F([E_{p_k,p_k},B])$). Thus, $C=0$ and we have shown that the conditions $$F([E_{p_1,p_1}-E_{p_i,p_i},B])=0$$ for $p_1\in\mathcal{P}$ and all $p_1\neq p_i\in\mathcal{P}$ imply that $$F([E_{p_i,p_i},B])=0$$ for all $p_i\in\mathcal{P}$. The other direction is clear.

\section{Appendix B}

In this Appendix, we prove Theorem~\ref{thm:h3bb} of Section~\ref{sec:toralpair}.  We break the proof into three Lemmas.  Lemma~\ref{lem:1and2} addresses the proof of Theorem~\ref{thm:h3bb} (i) and (ii), Lemma~\ref{lem:3and4} addresses the proof of Theorem~\ref{thm:h3bb} (iii) and (iv), and Lemma~\ref{lem:5and6} addresses the proof of Theorem 5 (v) and (vi).

\begin{lemma}\label{lem:1and2}
Each of the following pairs, consisting of a poset $\mathcal{P}$ and a functional $F_{\mathcal{P}}$, form a toral-pair $(\mathcal{P},F_{\mathcal{P}})$.
\begin{enumerate} [label=\textup(\roman*\textup)]
        \item $\mathcal{P}_3=\{p_1,p_2,p_3,p_4\}$ with $p_1\preceq p_2\preceq p_3,p_4$; $p_3\preceq p_5$; and $p_4\preceq p_6$, and $$F_{\mathcal{P}_3}=E^*_{p_1,p_5}+E^*_{p_1,p_6}+E^*_{p_2,p_3}+E^*_{p_2,p_4}+E^*_{p_2,p_6}.$$
        \item $\mathcal{P}^*_3=\{p_1,p_2,p_3,p_4\}$ with $p_1\preceq p_3$; $p_2\preceq p_4$; and $p_3,p_4\preceq p_5\preceq p_6$, and $$F_{\mathcal{P}^*_3}=E^*_{p_1,p_6}+E^*_{p_2,p_6}+E^*_{p_3,p_5}+E^*_{p_4,p_5}+E^*_{p_2,p_5}.$$
\end{enumerate}
\end{lemma}
\begin{proof}
We prove (i), as (ii) follows via a symmetric argument. To ease notation, let $\mathcal{P}=\mathcal{P}_3$ and $F=F_{\mathcal{P}_3}$. By definition, it is clear that $|Ext(\mathcal{P})|=3$ and $\Sigma(\mathcal{P})$ is contractible so that \textbf{(P1)} and \textbf{(P3)} of Definition~\ref{def:toralpair} are satisfied. Now, if $B\in\mathfrak{g}_A(\mathcal{P})\cap\ker(F)$, then $B$ must satisfy the following restrictions which are broken into 3 groups:

\bigskip
\noindent
\textbf{Group 1:}
\begin{itemize}
    \item $F([E_{p_5,p_5},B])=-E^*_{p_1,p_5}(B)=0$,
    \item $F([E_{p_1,p_3},B])=E^*_{p_3,p_5}(B)=0$,
    \item $F([E_{p_1,p_4},B])=E^*_{p_4,p_6}(B)=0$,
    \item $F([E_{p_3,p_3},B])=-E^*_{p_2,p_3}=0$,
    \item $F([E_{p_4,p_4},B])=-E^*_{p_2,p_4}=0$,
    \item $F([E_{p_3,p_3},B])=-E^*_{p_3,p_3}=0$,
    \item $F([E_{p_2,p_5},B])=-E^*_{p_1,p_2}(B)=0$,
    \item $F([E_{p_3,p_5},B])=-E^*_{p_1,p_3}(B)=0$.
\end{itemize}
\textbf{Group 2:}
\begin{itemize}
    \item $F([E_{p_1,p_1},B])=E^*_{p_1,p_5}(B)+E^*_{p_1,p_6}(B)=0$,
    \item $F([E_{p_2,p_2},B])=E^*_{p_2,p_3}(B)+E^*_{p_2,p_4}(B)+E^*_{p_2,p_6}(B)=0$,
    \item $F([E_{p_6,p_6},B])=-E^*_{p_1,p_6}(B)-E^*_{p_2,p_6}(B)=0$,
    \item $F([E_{p_1,p_2},B])=E^*_{p_2,p_5}(B)+E^*_{p_2,p_6}(B)=0$,
    \item $F([E_{p_4,p_6},B])=-E^*_{p_1,p_4}(B)-E^*_{p_2,p_4}(B)=0$.
\end{itemize}
\textbf{Group 3:}
\begin{itemize}
    \item $F([E_{p_1,p_5},B])=E^*_{p_5,p_5}(B)-E^*_{p_1,p_1}(B)=0$,
    \item $F([E_{p_1,p_6},B])=E^*_{p_6,p_6}(B)-E^*_{p_1,p_1}(B)=0$,
    \item $F([E_{p_2,p_3},B])=E^*_{p_3,p_3}(B)-E^*_{p_2,p_2}(B)=0$,
    \item $F([E_{p_2,p_4},B])=E^*_{p_4,p_6}(B)+E^*_{p_4,p_4}(B)-E^*_{p_2,p_2}(B)=0$,
    \item $F([E_{p_2,p_6},B])=-E^*_{p_1,p_2}(B)+E^*_{p_6,p_6}(B)-E^*_{p_2,p_2}(B)=0$.
\end{itemize}
The restrictions of Group 1 immediately imply that
\begin{equation}\label{eqn:h3bb1}
E^*_{p_1,p_2}(B)=E^*_{p_1,p_3}(B)=E^*_{p_1,p_5}(B)=E^*_{p_2,p_3}(B)=E^*_{p_2,p_4}(B)=E^*_{p_3,p_5}(B)=E^*_{p_4,p_6}(B)=0.
\end{equation}
Combining the restrictions of Group 1 with those of Group 2, we find that
\begin{equation}\label{eqn:h3bb2}
E^*_{p_1,p_4}(B)=E^*_{p_1,p_6}(B)=E^*_{p_2,p_5}(B)=E^*_{p_2,p_6}(B)=0.
\end{equation}
Finally, combining the restrictions of Groups 1 and 2 with the restrictions of Group 3, we find that 
\begin{equation}\label{eqn:h3bb3}
E^*_{p_i,p_i}(B)=E^*_{p_j,p_j}(B)\text{ for all }p_i,p_j\in\mathcal{P}.
\end{equation} 
Since $B\in \mathfrak{sl}(6)$, equation (\ref{eqn:h3bb3}) implies that $E^*_{p_i,p_i}(B)=0$ for all $p_i\in\mathcal{P}$. Thus, $B=0$ and $\mathfrak{g}_A(\mathcal{P})$ is Frobenius with a choice of Frobenius functional given by $F$.

Given the form of the Frobenius functional $F$, we have that
$F$ satisfies \textbf{(F1)} through \textbf{(F4)} of Definition~\ref{def:toralpair} as follows:
\begin{itemize}
    \item $F$ is clearly small,
    \item $D_{F}(\mathcal{P})=\{p_1,p_2\}$ forms an ideal in $\mathcal{P}$, $U_{F}(\mathcal{P})=\{p_3,p_4,p_5,p_6\}$ forms a filter, and $B_{F}(\mathcal{P})=\emptyset$,
    \item $\Gamma_F$ contains the only edges, $(p_1,p_5)$ and $(p_1,p_6)$, between elements of $Ext(\mathcal{P})$, and
    \item equations (\ref{eqn:h3bb1}), (\ref{eqn:h3bb2}), and (\ref{eqn:h3bb3}) show that $F$ satisfies \textbf{(F4)} of Definition~\ref{def:toralpair}.
\end{itemize}
It remains to show that \textbf{(P2)} is satisfied; that is, $\mathfrak{g}_A(\mathcal{P})$ has a spectrum consisting of an equal number of 0's and 1's. To determine the spectrum of $\mathfrak{g}_A(\mathcal{P})$, it suffices to calculate $[\widehat{F},x]$ for $x\in\mathscr{B}_{\mathcal{P},F}$. Note that $\mathscr{B}_{\mathcal{P},F}$ can be partitioned into two sets: 
$$G_0=\{E_{p_5,p_5}-E_{p_1,p_1},E_{p_6,p_6}-E_{p_1,p_1},E_{p_3,p_3}-E_{p_2,p_2},E_{p_4,p_4}-E_{p_2,p_2},E_{p_6,p_6}-E_{p_2,p_2},E_{p_1,p_2},E_{p_3,p_5},E_{p_4,p_6}\}$$
which consists of eigenvectors of $ad(\widehat{F})$ with eigenvalue 0, and
$$G_1=\{E_{p_1,p_5},E_{p_1,p_6},E_{p_1,p_3},E_{p_1,p_4},E_{p_2,p_5},E_{p_2,p_6},E_{p_2,p_3},E_{p_2,p_4}\}$$
which are eigenvectors with eigenvalue 1. As $|G_0|=|G_1|$, we conclude that $\mathfrak{g}_A(\mathcal{P})$ has a binary spectrum and $(\mathcal{P},F)$ forms a toral-pair.
\end{proof}

\begin{lemma}\label{lem:3and4}
Each of the following pairs, consisting of a poset $\mathcal{P}$ and a functional $F_{\mathcal{P}}$, form a toral-pair $(\mathcal{P},F_{\mathcal{P}})$.
\begin{enumerate} [label=\textup(\roman*\textup)]
    \setcounter{enumi}{2}
    \item $\mathcal{P}_{4,n}=\{p_1,\hdots,p_n\}$ with $p_1\preceq p_2\preceq\hdots\preceq p_{n-1}$ as well as $p_1\preceq p_2\preceq\hdots\preceq p_{\lfloor\frac{n}{2}\rfloor}\preceq p_n$, and $$F_{\mathcal{P}_{4,n}}=\sum_{i=1}^{\lfloor\frac{n-1}{2}\rfloor}E^*_{p_i,p_{n-i}}+\sum_{i=1}^{\lfloor\frac{n}{2}\rfloor}E^*_{p_i,p_n}.$$
    \item $\mathcal{P}^*_{4,n}=\{p_1,\hdots,p_n\}$ with $p_1\preceq p_2\preceq \hdots\preceq p_{\lfloor\frac{n-1}{2}\rfloor}\preceq \lceil\frac{n}{2}\rceil\preceq\hdots\preceq p_{n}$ as well as $p_{\lfloor\frac{n-1}{2}\rfloor+1}\preceq p_{\lfloor\frac{n-1}{2}\rfloor+2}\preceq \hdots\preceq p_{n}$, and $$F_{\mathcal{P}^*_{4,n}}=\sum_{i=1}^{\lfloor\frac{n-1}{2}\rfloor}E^*_{p_i,p_{n+1-i}}+\sum_{i=\lceil\frac{n}{2}\rceil}^{n}E^*_{p_{\lceil\frac{n}{2}\rceil},p_i}.$$
\end{enumerate}
\end{lemma}
\begin{proof}
We prove (i), as (ii) follows via a symmetric argument. To ease notation, let $\mathcal{P}=\mathcal{P}_{4,n}$ and $F=F_{\mathcal{P}_{4,n}}$. By definition, it is clear that $|Ext(\mathcal{P})|=3$ so that \textbf{(P1)} of Definition~\ref{def:toralpair} is satisfied. To see that \textbf{(P3)} of Definition~\ref{def:toralpair} is satisfied, i.e., $\Sigma(\mathcal{P})$ is contractible, note that $\Sigma(\mathcal{P})$ is formed by adjoining a $\lceil\frac{n}{2}\rceil$-simplex to a $(n-1)$-simplex along a face; such a space is star-convex and thus contractible. Now, if $B\in\mathfrak{g}_A(\mathcal{P})\cap\ker(F)$, then $B$ must satisfy the following restrictions which are broken into 3 groups:

\bigskip
\noindent
\textbf{Group 1:}
\begin{itemize}
    \item $F([E_{p_{\lfloor\frac{n}{2}\rfloor},p_{\lfloor\frac{n}{2}\rfloor}},B])=E^*_{p_{\lfloor\frac{n}{2}\rfloor},p_n}(B)=0$ (for $n$ even),
    \item $F([E_{p_{n-i},p_{n-i}},B])=-E^*_{p_i,p_{n-i}}(B)=0$ for $1\le i\le \lfloor\frac{n-1}{2}\rfloor$,
    \item $F([E_{p_i,p_{n-j}},B])=E^*_{p_{n-j},p_{n-i}}(B)=0$ for $1\le i<j\le \lfloor\frac{n}{2}\rfloor$,
    \item $F([E_{p_i,p_{n-j}},B])=-E^*_{p_j,p_i}(B)=0$ for $1\le j<i\le \lfloor\frac{n}{2}\rfloor$,
    \item $F([E_{p_{n-j},p_{n-i}},B])=-E^*_{p_i,p_{n-j}}(B)=0$ for $1\le i<j\le\lfloor\frac{n}{2}\rfloor$.
\end{itemize}
\textbf{Group 2:}
\begin{itemize}
    \item $F([E_{p_i,p_i},B])=E^*_{p_i,p_{n-i}}(B)+E^*_{p_i,p_n}=0$ for $1\le i\le \lfloor\frac{n-1}{2}\rfloor$,
    \item $F([E_{p_i,p_j},B])=E^*_{p_j,p_{n-i}}(B)+E^*_{p_j,p_n}(B)=0$ for $1\le i<j\le \lfloor\frac{n}{2}\rfloor$.
\end{itemize}
\textbf{Group 3:}
\begin{itemize}
    \item $F(E_{p_i,p_{n-i}},B])=E^*_{p_{n-i},p_{n-i}}(B)-E^*_{p_i,p_i}(B)=0$ for $1\le i\le \lfloor\frac{n-1}{2}\rfloor$,
    \item $F([E_{p_i,p_n},B])=E^*_{p_n,p_n}(B)-E^*_{p_i,p_i}(B)-\sum_{1\le j<i}E^*_{p_j,p_i}(B)=0$ for $1\le i\le \lfloor\frac{n}{2}\rfloor$.
\end{itemize}
The restrictions of Group 1 immediately imply that
\begin{equation}\label{eqn:hnbb11}
E^*_{p_{\lfloor\frac{n}{2}\rfloor},p_n}(B)=0\text{ for }n\text{ even},
\end{equation}
\begin{equation}\label{eqn:hnbb12}
E^*_{p_i,p_{n-i}}(B)=0\text{ for }1\le i\le \bigg\lfloor\frac{n-1}{2}\bigg\rfloor,
\end{equation}
\begin{equation}\label{eqn:hnbb13}
E^*_{p_{n-j},p_{n-i}}(B)=0\text{ for }1\le i<j\le \bigg\lfloor\frac{n}{2}\bigg\rfloor,
\end{equation}
\begin{equation}\label{eqn:hnbb14}
E^*_{p_j,p_i}(B)=0\text{ for }1\le j<i\le \bigg\lfloor\frac{n}{2}\bigg\rfloor,
\end{equation}
and
\begin{equation}\label{eqn:hnbb15}
E^*_{p_i,p_{n-j}}(B)=0\text{ for }1\le i<j\le\bigg\lfloor\frac{n}{2}\bigg\rfloor.
\end{equation}
Combining the restrictions of Group 1 with those of Group 2, we can conclude that 
\begin{equation}\label{eqn:hnbb16}
E^*_{p_i,p_n}(B)=0\text{ for }1\le i\le \bigg\lfloor\frac{n-1}{2}\bigg\rfloor
\end{equation}
and 
\begin{equation}\label{eqn:hnbb17}
E^*_{p_j,p_{n-i}}(B)=0\text{ for }1\le i<j\le \bigg\lfloor\frac{n}{2}\bigg\rfloor.
\end{equation}
Finally, combining the restrictions of Groups 1 and 2 with those of Group 3, we find
that 
\begin{equation}\label{eqn:hnbb18}
E^*_{p_i,p_i}(B)=E^*_{p_j,p_j}(B)\text{ for }p_i,p_j\in\mathcal{P}.
\end{equation}
Since $B\in\mathfrak{sl}(n)$, equation (\ref{eqn:hnbb18}) implies that $E^*_{p_i,p_i}(B)=0$ for all $i\in\mathcal{P}$. Thus, $B=0$ and $\mathfrak{g}_A(\mathcal{P})$ is Frobenius with a choice of Frobenius functional given by $F$.

Given the form of the Frobenius functional $F$, we have that
$F$ satisfies \textbf{(F1)} through \textbf{(F4)} of Definition~\ref{def:toralpair} as follows:
\begin{itemize}
    \item $F$ is clearly small,
    \item $D_{F}(\mathcal{P})=\{p_i~|~1\le i\le \lfloor\frac{n}{2}\rfloor\}$ forms an ideal in $\mathcal{P}$, $U_{F}(\mathcal{P})=\{p_i~|~\lfloor\frac{n}{2}\rfloor< i\le n \}$ forms a filter, and $B_{F}(\mathcal{P})=\emptyset$,
    \item $\Gamma_F$ contains the only edges, $(p_1,p_{n-1})$ and $(p_1,p_n)$, between elements of $Ext(\mathcal{P})$, and
    \item equations (\ref{eqn:hnbb11}) through (\ref{eqn:hnbb18}) show that $F$ satisfies \textbf{(F4)} of Definition~\ref{def:toralpair}.
\end{itemize}
It remains to show that \textbf{(P2)} is satisfied; that is, $\mathfrak{g}_A(\mathcal{P})$ has a spectrum consisting of an equal number of 0's and 1's. To determine the spectrum of $\mathfrak{g}_A(\mathcal{P})$, it suffices to calculate $[\widehat{F},x]$ for $x\in\mathscr{B}_{\mathcal{P},F}$. Note that the elements of $\mathscr{B}_{\mathcal{P},F}$ can be partitioned as follows: $$S_1=\{E_{p_i,p_j}~|~1\preceq i\prec j\preceq \bigg\lfloor\frac{n}{2}\bigg\rfloor\},$$
$$S_2=\{E_{p_i,p_{n-j}}~|~1\le i,j\le \bigg\lfloor\frac{n}{2}\bigg\rfloor\text{ and }i\neq n-j\},$$
$$S_3=\{E_{p_{n-j},p_{n-i}}~|~1\le i<j\le \bigg\lfloor\frac{n-1}{2}\bigg\rfloor\},$$
$$S_4=\{E_{p_i,p_n}~|~1\preceq i\preceq \bigg\lfloor\frac{n}{2}\bigg\rfloor\},$$
$$S_5=\{E_{p_i,p_i}-E_{p_n,p_n}~|~1\le i<n\},$$ where 

$$|S_1| = \sum_{i=1}^{\lfloor\frac{n}{2}\rfloor}\bigg(\bigg\lfloor\frac{n}{2}\bigg\rfloor-i\bigg)=  \begin{cases} 
      \frac{n^2-4n+3}{8}, & n\text{ odd}; \\
                                                &                        \\
      \frac{n^2-2n}{8}, & n\text{ even},
   \end{cases}$$

$$|S_2| = \begin{cases} 
      \lfloor\frac{n-1}{2}\rfloor, & n\text{ odd}; \\
                                                &                        \\
      \lfloor\frac{n-1}{2}\rfloor(\lfloor\frac{n-1}{2}\rfloor-1), & n\text{ even},
   \end{cases} = \begin{cases} 
     \frac{n^2-2n+1}{4}, & n\text{ odd}; \\
                                                &                        \\
      \frac{n^2-2n}{4}, & n\text{ even},
   \end{cases}$$

$$|S_3| = \sum_{i=1}^{\lfloor\frac{n-1}{2}\rfloor}\bigg(\bigg\lfloor\frac{n-1}{2}\bigg\rfloor-i\bigg)= \begin{cases} 
      \frac{n^2-4n+3}{8}, & n\text{ odd}; \\
                                                &                        \\
      \frac{n^2-6n+8}{8}, & n\text{ even},
   \end{cases}$$

$$|S_4| = \bigg\lfloor\frac{n}{2}\bigg\rfloor = \begin{cases} 
      \frac{n-1}{2}, & n\text{ odd}; \\
                                                &                        \\
      \frac{n}{2}, & n\text{ even},
   \end{cases}$$  and 

$$|S_5|=n-1.$$ Furthermore, note that the elements contained in $G_0=S_1\cup S_3\cup S_5$ are eigenvalues of $ad(\widehat{F})$ with eigenvalue 0, while the elements contained in $G_1=S_2\cup S_4$ are eigenvectors with eigenvalue 1. Thus, as $\lfloor\frac{n}{2}\rfloor=\lfloor\frac{n-1}{2}\rfloor=\frac{n-1}{2}$ when $n$ is odd and $\lfloor\frac{n-1}{2}\rfloor=\frac{n-2}{2}$ when $n$ is even,
\[|G_0|=|G_1| =  \begin{cases} 
      \frac{n^2-1}{4}, & n\text{ odd}; \\
                                                &                        \\
      \frac{n^2}{4}, & n\text{ even}.
   \end{cases}
\]
Therefore, $\mathfrak{g}_A(\mathcal{P})$ has a binary spectrum and $(\mathcal{P},F)$ is a toral-pair.
\end{proof}

\begin{lemma}\label{lem:5and6}
Each of the following pairs, consisting of a poset $\mathcal{P}$ and a functional $F_{\mathcal{P}}$, form a toral-pair $(\mathcal{P},F_{\mathcal{P}})$.
\begin{enumerate} [label=\textup(\roman*\textup)]
    \setcounter{enumi}{4}
    \item $\mathcal{P}_{5,n}=\{p_1,\hdots,p_{2n+1}\}$ with $p_i\preceq p_j$ for $1\le i<2n$ odd and $i+1\le j\le 2n+1$ as well as $p_i\preceq p_j$ for $1<i<2n$ even and $i+2\le j\le 2n+1$, and $$F_{\mathcal{P}_{5,n}}=E^*_{p_1,p_{2n+1}}+\sum_{i=1}^{\lceil\frac{n-1}{2}\rceil}E^*_{p_i, p_{2n}}+\sum_{k=1}^{\lfloor\frac{n-1}{2}\rfloor}E^*_{p_{2k}, p_{2n-2k}}+\sum_{k=1}^{\lfloor\frac{n-1}{2}\rfloor}E^*_{p_{2k+1}, p_{2n-2k+1}}.$$
    \item $\mathcal{P}^*_{5,n}=\{p_1,\hdots,p_{2n+1}\}$ with $p_i\preceq p_j$ for $1\le i<2n$ odd and $i+2\le j\le 2n+1$, as well as $p_i\preceq p_j$ for $1< i<2n$ even and $i+1\le j\le 2n+1$, and $$F_{\mathcal{P}^*_{5,n}}=E^*_{p_1,p_{2n+1}}+\sum_{i=2\lfloor\frac{n+1}{4}\rfloor+1}^{2n+1}E^*_{p_2, p_i}+\sum_{k=2}^{\lfloor\frac{n+1}{2}\rfloor}E^*_{p_{2k}, p_{2n-2k+4}}+\sum_{k=1}^{\lfloor\frac{n-1}{2}\rfloor}E^*_{p_{2k+1}, p_{2n-2k+1}}.$$
\end{enumerate}
\end{lemma}
\begin{proof}
We prove (i), as (ii) follows via a symmetric argument. To ease notation, let $\mathcal{P}=\mathcal{P}_{5,n}$ and $F=F_{\mathcal{P}_{5,n}}$. By definition, it is clear that $|Ext(\mathcal{P})|=3$ so that \textbf{(P1)} of Definition~\ref{def:toralpair} is satisfied. To see that \textbf{(P3)} of Definition~\ref{def:toralpair} is satisfied, i.e., $\Sigma(\mathcal{P})$ is contractible, note that $\Sigma(\mathcal{P})$ is formed by taking $n$ suspensions of a point; since the suspension of a contractible simplicial complex is contractible, the result follows. Now, if $B\in\mathfrak{g}_A(\mathcal{P})\cap\ker(F)$, then $B$ must satisfy the following restrictions which are broken into six groups:

\bigskip
\noindent
\textbf{Group 1:}
\begin{itemize}
    \item $F([E_{p_i,p_{2n+1}},B])=-E^*_{p_1,p_i}(B)=0$ for $1<i<2n$,
    \item $F([E_{p_{2n+1},p_{2n+1}},B])=-E^*_{p_1,p_{2n+1}}(B)=0$,
    \item $F([E_{p_1,p_1},B])=E^*_{p_1,p_{2n}}(B)+E^*_{p_1,p_{2n+1}}(B)=0$.
\end{itemize}
\textbf{Group 2:}
\begin{itemize}
    \item $F([E_{p_{2n-2k},p_{2n-2k}},B])=-E^*_{p_{2k},p_{2n-2k}}(B)=0$ for $1\le k\le \lfloor\frac{n-1}{2}\rfloor$,
    \item $F([E_{p_{2n-2k+1},p_{2n-2k+1}},B])=-E^*_{p_{2k+1},p_{2n-2k+1}}(B)=0$ for $1\le k\le \lfloor\frac{n-1}{2}\rfloor$.
\end{itemize}
\textbf{Group 3:} 
\begin{itemize}
    \item $F([E_{p_{2k},p_{2k}},B])=E^*_{p_{2k},p_{2n-2k}}(B)+E^*_{p_{2k},p_{2n}}(B)=0$ for $1\le k\le \lfloor\frac{n-1}{2}\rfloor$,
    \item $F([E_{p_{2k+1},p_{2k+1}},B])=E^*_{p_{2k+1},p_{2n-2k+1}}(B)+E^*_{p_{2k+1},p_{2n}}(B)=0$ for $1\le k\le \lfloor\frac{n-1}{2}\rfloor$,
    \item $F([E_{p_{2k},p_{2k}},B])=E^*_{p_{2k},p_{2n}}(B)$ for $n=\lceil\frac{n-1}{2}\rceil$ and $n$ even,
    \item $F([E_{p_{2k+1},p_{2k+1}},B])=E^*_{p_{2k+1},p_{2n}}(B)$ for $n=\lceil\frac{n-1}{2}\rceil$ and $n$ even,
    \item $F([E_{p_{2k+1}, p_{2n-2k}},B])=E^*_{p_{2n-2k}, p_{2n}}(B)=0$ for $1\le k\le \lfloor\frac{n-1}{2}\rfloor$,
    \item $F([E_{p_{2k}, p_{2n-2k+1}},B])=E^*_{p_{2n-2k+1}, p_{2n}}(B)=0$ for $1\le k\le \lfloor\frac{n-1}{2}\rfloor$,
    \item $F([E_{p_1,p_i},B])=E^*_{p_i,p_{2n}}(B)+E^*_{p_i,p_{2n+1}}(B)=0$ for $1<i<2n$.
\end{itemize}
\textbf{Group 4:}
\begin{itemize}
    \item $F([E_{p_{2k},p_{2n-2k_1}},B])=E^*_{p_{2n-2k_1}, p_{2n}}(B)-E^*_{p_{2k_1}, p_{2k}}(B)=0$ for $1\le k_1<k\le\lceil\frac{n-1}{2}\rceil$,
    \item $F([E_{p_{2k},p_{2n-2k_1+1}},B])=E^*_{p_{2n-2k_1+1}, p_{2n}}(B)-E^*_{p_{2k_1+1}, p_{2k}}(B)=0$ for $1\le k_1<k\le\lceil\frac{n-1}{2}\rceil$,
    \item $F([E_{p_{2k+1}, p_{2n-2k_1}},B])=E^*_{p_{2n-2k_1}, p_{2n}}(B)-E^*_{p_{2k_1}, p_{2k+1}}(B)=0$ for $1\le k_1<k\le\lceil\frac{n-1}{2}\rceil$,
    \item $F([E_{p_{2k+1},p_{2n-2k_1+1}},B])=E^*_{p_{2n-2k_1+1}, p_{2n}}(B)-E^*_{p_{2k_1+1}, p_{2k+1}}(B)=0$ for $1\le k_1<k\le\lceil\frac{n-1}{2}\rceil$.
\end{itemize}
\textbf{Group 5:}
\begin{itemize}
    \item $F([E_{p_{2n-2k},p_{2n-2k_1}},B])=-E^*_{p_{2k_1},p_{2n-2k}}(B)=0$ for $1\le k_1<k\le\lfloor\frac{n-1}{2}\rfloor$,
    \item $F([E_{p_{2n-2k},p_{2n-2k_1+1}},B])=-E^*_{p_{2k_1+1},p_{2n-2k}}(B)=0$ for $1\le k_1<k\le\lfloor\frac{n-1}{2}\rfloor$,
    \item $F([E_{p_{2n-2k+1},p_{2n-2k_1}},B])=-E^*_{p_{2k_1},p_{2n-2k+1}}(B)=0$ for $1\le k_1<k\le\lfloor\frac{n-1}{2}\rfloor$,
    \item $F([E_{p_{2n-2k+1},p_{2n-2k_1+1}},B])=-E^*_{p_{2k_1+1},p_{2n-2k+1}}(B)=0$ for $1\le k_1<k\le\lfloor\frac{n-1}{2}\rfloor$,
    \item $F([E_{p_{2k},p_i},B])=E^*_{p_i, p_{2n-2k}}(B)+E^*_{p_i, p_{2n}}(B)=0$ for $1\le k\le\lfloor\frac{n-1}{2}\rfloor$ and $2k+1<i<2n-2k$,
    \item $F([E_{p_{2k+1},p_i},B])=E^*_{p_i, p_{2n-2k+1}}(B)+E^*_{p_i, p_{2n}}(B)=0$ for $1\le k\le\lfloor\frac{n-1}{2}\rfloor$ and $2k+1<i<2n-2k$,
    \item $F([E_{p_i, p_{2n}},B])=-\sum_{1\le j\le 2\lceil\frac{n-1}{2}\rceil+1}E^*_{p_j, p_i}(B)=0$ for $2\lceil\frac{n-1}{2}\rceil+1< i<2n$.
\end{itemize}
\textbf{Group 6:}
\begin{itemize}
    \item $F([E_{p_1,p_i},B])=E^*_{p_i,p_i}(B)-E^*_{p_1,p_1}(B)=0$ for $i=2n, 2n+1$,
    \item $F([E_{p_{2k}, p_{2n-2k}},B])=E^*_{p_{2n-2k}, p_{2n-2k}}(B)-E^*_{p_{2k}, p_{2k}}(B)+E^*_{p_{2n-2k}, p_{2n}}(B)=0$ for $1\le k\le\lfloor\frac{n-1}{2}\rfloor$,
    \item $F([E_{p_{2k+1}, p_{2n-2k+1}},B])=E^*_{p_{2n-2k+1}, p_{2n-2k+1}}(B)-E^*_{p_{2k+1}, p_{2k+1}}(B)+E^*_{p_{2n-2k+1}, p_{2n}}(B)=0$ for $1\le k\le\lfloor\frac{n-1}{2}\rfloor$,
    \item $F([E_{p_i, p_{2n}},B])=E^*_{p_{2n}, p_{2n}}(B)-E^*_{p_i,p_i}(B)-\sum_{1\le j<i}E^*_{p_j, p_i}(B)=0$ for $1\le i\le2\lceil\frac{n-1}{2}\rceil+1$.
\end{itemize}
Group 1 immediately implies that 
\begin{equation}\label{eqn:hnbb21}
E^*_{p_1,p_i}(B)=0\text{ for }1<i\le 2n+1.
\end{equation}
Next, Group 2 immediately implies that
\begin{equation}\label{eqn:hnbb22}
E^*_{p_{2k},p_{2n-2k}}(B)=E^*_{p_{2k+1},p_{2n-2k+1}}(B)=0\text{ for }1\le k\le \bigg\lfloor\frac{n-1}{2}\bigg\rfloor.
\end{equation}
Combining the restrictions of Groups 1 and 2 with those of Group 3, we may conclude that 
\begin{equation}\label{eqn:hnbb23}
E^*_{p_i,p_{2n}}(B)=E^*_{p_i,p_{2n+1}}(B)=0\text{ for }1\le i<2n.
\end{equation}
The restrictions of Groups 1 and 3 applied to those of Group 4, allow us to conclude that
\begin{equation}\label{eqn:hnbb24}
E^*_{p_i,p_{2k}}(B)=E^*_{p_i,p_{2k+1}}(B)=0\text{ for }1\le k\le \bigg\lceil\frac{n-1}{2}\bigg\rceil\text{ and }i<2k,2k+1.
\end{equation}
Applying the restrictions of Groups 1, 2, 3, and 4 to the restrictions of Group 5, implies that 
\begin{equation}\label{eqn:hnbb25}
E^*_{p_i,p_{2n-2k}}(B)=E^*_{p_i,p_{2n-2k+1}}(B)=0\text{ for }1\le k\le \bigg\lceil\frac{n-1}{2}\bigg\rceil\text{ and }1\le i<2n-2k.
\end{equation}
Note that the restrictions 1 through 6 of Group 5, when applied with the restrictions of the previous Groups, imply that
\begin{equation}\label{eqn:hnbb26}
E^*_{p_i,p_{2n-2k}}(B)=0\text{ for }1\le k\le \bigg\lceil\frac{n-1}{2}\bigg\rceil\text{ and }1\le i\neq 2k+1<2n-2k
\end{equation}
and
\begin{equation}\label{eqn:hnbb27}
E^*_{p_i,p_{2n-2k+1}}(B)=0\text{ for }1\le k\le \bigg\lceil\frac{n-1}{2}\bigg\rceil\text{ and }1\le i\neq 2k<2n-2k
\end{equation}
Combining restrictions 1 through 6 of Group 5 with restriction 7 we find that 
\begin{equation}\label{eqn:hnbb28}
E^*_{p_{2k+1},p_{2n-2k}}(B)=E^*_{p_{2k},p_{2n-2k+1}}(B)\text{ for }1\le k\le \bigg\lceil\frac{n-1}{2}\bigg\rceil.
\end{equation}
Finally, combining all of the above restrictions with the restrictions of Group 6 implies that
\begin{equation}\label{eqn:hnbb29}
E^*_{p_i,p_i}(B)=E^*_{p_j,p_j}(B)\text{ for }1\le i,j\le 2n+1.
\end{equation}
Since $B\in\mathfrak{sl}(2n+1)$, equation (\ref{eqn:hnbb29}) implies that $E^*_{p_i,p_i}(B)=0$ for all $1\le i\le 2n+1$. Thus, $B=0$ and $\mathfrak{g}_A(\mathcal{P})$ is Frobenius with a choice of Frobenius functional given by $F$.

Given the form of the Frobenius functional $F$, we have that
$F$ satisfies \textbf{(F1)} through \textbf{(F4)} of Definition~\ref{def:toralpair} as follows:
\begin{itemize}
    \item $F$ is clearly small,
    \item $D_{F}(\mathcal{P})=\{p_i~|~1\le i\le 2\lceil\frac{n-1}{2}\rceil+1\}$ forms an ideal in $\mathcal{P}$, $U_{F}(\mathcal{P})=\{p_i~|~2\lceil\frac{n-1}{2}\rceil+1<i\le 2n+1\}$ forms a filter, and $B_{F}(\mathcal{P})=\emptyset$,
    \item $\Gamma_F$ contains the only edges, $(p_1,p_{2n})$ and $(p_1,p_{2n+1})$, between elements of $Ext(\mathcal{P})$, and
    \item equations (\ref{eqn:hnbb21}) through (\ref{eqn:hnbb29}) show that $F$ satisfies \textbf{(F4)} of Definition~\ref{def:toralpair}.
\end{itemize}
It remains to show that \textbf{(P2)} is satisfied; that is, $\mathfrak{g}_A(\mathcal{P})$ has a spectrum consisting of an equal number of 0's and 1's. To determine the spectrum of $\mathfrak{g}_A(\mathcal{P})$, it suffices to calculate $[\widehat{F},x]$ for $x\in\mathscr{B}_{\mathcal{P},F}$. Note that the elements of $\mathscr{B}_{\mathcal{P},F}$ can be partitioned as follows: $$S^1_1=\{E_{p_{2n+1},p_{2n+1}}-E_{p_1,p_1},E_{p_{2n},p_{2n}}-E_{p_1,p_1}\},$$ $$S^2_1=\{E_{p_{2n},p_{2n}}-E_{p_{2k},p_{2k}},E_{p_{2n},p_{2n}}-E_{p_{2k+1},p_{2k+1}}~|~1\le k\le\bigg\lceil\frac{n-1}{2}\bigg\rceil\},$$ $$S^3_1= \{E_{p_{2n-2k},p_{2n-2k}}-E_{p_{2k},p_{2k}},E_{p_{2n-2k+1},p_{2n-2k+1}}-E_{p_{2k},p_{2k}}~|~1\le k\le\bigg\lfloor\frac{n-1}{2}\bigg\rfloor\},$$ $$S_1=S_1^1~\cup~S_1^2~\cup~S_1^3,$$
$$S_2=\{E_{p_i,p_j}~|~1\preceq i\preceq 2\bigg\lceil\frac{n-1}{2}\bigg\rceil+1,2\bigg\lceil\frac{n-1}{2}\bigg\rceil+1\prec j\preceq 2n+1\},$$
$$S_3=\{E_{p_i,p_j}~|~1\preceq i\prec j\preceq 2\bigg\lceil\frac{n-1}{2}\bigg\rceil+1\},$$
$$S_4=\{E_{p_i,p_j}~|~2\bigg\lceil\frac{n-1}{2}\bigg\rceil+1\prec i\prec j\preceq 2n+1\},$$ where $$|S_1|=2n,$$ 

$$|S_2| =  \begin{cases} 
      n^2+n, & n\text{ odd}; \\
      n^2+n, & n\text{ even};
   \end{cases}$$

$$|S_3| =  \begin{cases} 
      \frac{n^2-2n+1}{2}, & n\text{ odd}; \\
      \frac{n^2}{2}, & n\text{ even};
   \end{cases}$$
and
$$|S_4| =  \begin{cases} 
      \frac{n^2-1}{2}, & n\text{ odd}; \\
      \frac{n^2-2n}{2}, & n\text{ even}.
   \end{cases}$$
Furthermore, note that the elements contained in $G_0=S_1\cup S_3\cup S_5$ are eigenvectors of $ad(\widehat{F})$ with eigenvalue 0, while the elements of $G_1=S_2$ are eigenvectors with eigenvalue 1. Therefore,
\[|G_0|=|G_1| =  \begin{cases} 
      n^2+n, & n\text{ odd}; \\
                                                &                        \\
       n^2+n, & n\text{ even}.
   \end{cases}
\]
Therefore, $\mathfrak{g}_A(\mathcal{P})$ has a binary spectrum and $(\mathcal{P},F)$ is a toral-pair.
\end{proof}

\end{document}